\documentclass[a4paper,12pt]{article}
\usepackage{latexsym}
\usepackage{amsmath}
\usepackage{amssymb}
\usepackage{enumerate}
\usepackage{epic,eepic}

\usepackage{theorem}
\newtheorem{theorem}{Theorem}[section]
\newtheorem{proposition}[theorem]{Proposition}
\newtheorem{lemma}[theorem]{Lemma}
\newtheorem{corollary}[theorem]{Corollary}
\newtheorem{claim}[theorem]{Claim}

\theorembodyfont{\rmfamily}
\newtheorem{proof}{\textmd{\textit{Proof.}}}

\newtheorem{remark}[theorem]{Remark}

\newtheorem{definition}[theorem]{Definition}
\newtheorem{step}{Step}

\makeatletter

\@addtoreset{equation}{section}
\makeatother

\newcommand{\qedd}{\hfill \Box}
\newcommand{\ve}{\varepsilon}
\newcommand{\del}{\partial}
\newcommand{\lra}{\longrightarrow}
\newcommand{\grad}{\nabla\!\!_-}

\newcommand{\R}{\ensuremath{\mathbb{R}}}

\newcommand{\cC}{\ensuremath{\mathcal{C}}}
\newcommand{\cF}{\ensuremath{\mathcal{F}}}
\newcommand{\cK}{\ensuremath{\mathcal{K}}}
\newcommand{\cL}{\ensuremath{\mathcal{L}}}
\newcommand{\cP}{\ensuremath{\mathcal{P}}}
\newcommand{\cR}{\ensuremath{\mathcal{R}}}
\newcommand{\cS}{\ensuremath{\mathcal{S}}}
\newcommand{\cT}{\ensuremath{\mathcal{T}}}
\newcommand{\cV}{\ensuremath{\mathcal{V}}}
\newcommand{\bS}{\ensuremath{\mathbf{S}}}

\def\id{\mathop{\mathrm{id}}\nolimits}
\def\inj{\mathop{\mathrm{inj}}\nolimits}
\def\loc{\mathop{\mathrm{loc}}\nolimits}
\def\ac{\mathop{\mathrm{ac}}\nolimits}
\def\supp{\mathop{\mathrm{supp}}\nolimits}
\def\div{\mathop{\mathrm{div}}\nolimits}

\def\Ric{\mathop{\mathrm{Ric}}\nolimits}
\def\Ent{\mathop{\mathrm{Ent}}\nolimits}

\title{Non-contraction of heat flow on Minkowski spaces}
\author{Shin-ichi Ohta\thanks{Department of Mathematics, Kyoto University,
Kyoto 606-8502, Japan ({\sf sohta@math.kyoto-u.ac.jp}) \&
Max-Planck-Intitut f\"ur Mathematik, Vivatsgasse 7, 53111 Bonn, Germany.
Supported in part by the Grant-in-Aid for Young Scientists (B) 20740036.}
and Karl-Theodor Sturm\thanks{Institut f\"ur Angewandte Mathematik,
Universit\"at Bonn, Endenicher Allee 60, 53115 Bonn, Germany
({\sf sturm@uni-bonn.de}).}}
\date{}
\begin{document}

\maketitle

\begin{abstract}
We study contractivity properties of gradient flows for functions
on normed spaces or, more generally, on Finsler manifolds.
Contractivity of the flows turns out to be equivalent to a new notion
of convexity for the functions.
This is different from the usual convexity along geodesics in non-Riemannian Finsler manifolds.
As an application, we show that the heat flow on Minkowski normed spaces
other than inner product spaces is not contractive with respect to
the quadratic Wasserstein distance.
\end{abstract}

\section{Introduction}

The main goal of this article is to prove that, for the heat flow on
a Minkowski normed space, no bound for the exponential growth
of the $L^2$-Wasserstein distance exists, unless the space
is an inner product space.
This is rather surprising, in particular, in view of the fact that the heat flow
is the gradient flow in the $L^2$-Wasserstein space $\cP_2$
of the relative entropy and the fact that
the latter is known to be a convex function on $\cP_2$.
In order to find an explanation for this phenomenon, we will
first of all study the contraction of the gradient flow of
a function on a Finsler manifold.
A Finsler manifold is a manifold carrying a Minkowski norm on
each tangent space, instead of an inner product for Riemannian manifolds.
A Minkowski norm is a generalization of usual norms,
and is not necessarily centrally symmetric.
We will always assume that a Minkowski norm is strongly convex
(and in particular strictly convex, see Subsection~\ref{ssc:Mpre} for the definition).

In Riemannian manifolds, given $K \in \R$, it is well-known that
the $K$-convexity of a function $f$ along geodesics $\gamma$
(i.e., $(f \circ \gamma)'' \ge K|\dot{\gamma}|^2$ in the weak sense)
implies the \emph{$K$-contraction} of the gradient flow of $f$, namely
\[ d\big( \xi(t),\zeta(t) \big) \le e^{-Kt} d\big( \xi(0),\zeta(0) \big) \]
holds for all $t \ge 0$ and $\xi,\zeta$ solving $\dot{\xi}(t)=\nabla(-f)(\xi(t))$,
$\dot{\zeta}(t)=\nabla(-f)(\zeta(t))$.
This is obtained via the first variation formulas for the distance
$d(\xi(t),\zeta(t))$ and the function $f$.
In Finsler manifolds or even in strictly convex normed spaces, however,
it has been unclear whether the gradient flows of convex functions
are contractive (cf.\ \cite[Introduction]{AGS}).
To avoid trivial counter-examples, of course strict convexity must be imposed
(see Remark~\ref{rm:linf}).

The point is that, although the aforementioned first variation formulas do exist
also in the Finsler setting, they use different approximate inner products
(see the paragraph following Definition~\ref{df:bicon}).
Keeping this in mind, we introduce a new notion of convexity,
called the \emph{skew convexity}, which is equivalent to the usual convexity
in Riemannian manifolds.
We show that the $K$-skew convexity of a function on a Finsler manifold is equivalent
to the $K$-contraction of its gradient flow (Theorem~\ref{th:equiv}).
A difference between the skew convexity and the convexity along geodesics
is observed by considering distance functions.
In Minkowski spaces, the squared norm is always $2$-skew convex,
while in general it is only $K$-convex for some $K \ge 0$.
We also construct an explicit example of a convex function which is not $0$-skew convex.
This negatively answers the above question in \cite{AGS} (see Section~\ref{sc:dist} for details).

In the second part of the article, we apply our technique to
the heat flow on Minkowski spaces.
Due to the celebrated work of Jordan et al~\cite{JKO},
the heat flow on Euclidean spaces can be regarded as the gradient flow
of the relative entropy in the $L^2$-Wasserstein space.
This provides a somewhat geometric interpretation of the non-expansion
($0$-contraction) of heat flow with respect to the Wasserstein distance,
as the relative entropy is known to be convex along Wasserstein geodesics
(also called \emph{displacement convex}, \cite{Mc}).
More generally, on Riemannian manifolds, both the $K$-convexity
of the relative entropy and the $K$-contraction of heat flow are
equivalent to the lower Ricci curvature bound $\Ric \ge K$ (\cite{vRS}).
Note that the Wasserstein space over a Riemannian manifold
possesses a sort of Riemannian structure, for which the first variation formulas
are available (see \cite{Ot}, \cite{AGS}, \cite{Vi}, \cite{Er}).
We also remark that Gigli~\cite{Gi} recently showed the uniqueness
of the gradient flow of the relative entropy (with respect to a probability measure)
for metric measure spaces such that the relative entropy is
$K$-convex for some $K \in \R$, without relying on the contractivity.

In our previous works~\cite{Oint}, \cite{OS},
we have extended the equivalence between the Ricci curvature bound
and the convexity of the relative entropy, as well as
the identification of (nonlinear) heat flow with the gradient flow
of the relative entropy with respect to the reverse Wasserstein distance,
to Finsler manifolds.
In particular, the relative entropy on any Minkowski space is convex
(see also \cite[page 908]{Vi}).
Then it is natural to ask whether the heat flow on Minkowski spaces
is contractive or not (see also the fourth comment in \cite[Section~5]{Gi}).
Our main result gives a complete answer to this question.

\begin{theorem}\label{th:nonc}
The heat flow on a Minkowski normed space $(\R^n,\|\cdot\|)$ is not $K$-contractive
with respect to the reverse $L^2$-Wasserstein distance
for any $K \in \R$, unless $(\R^n,\|\cdot\|)$ is an inner product space.
\end{theorem}
Our proof uses a geometric characterization of inner products among
Minkowski norms (Claims~\ref{cl:nonc}, \ref{cl:tri}).

Theorem~\ref{th:nonc} means that the Wasserstein contraction
implies that the space must be Riemannian.
This makes a contrast with the aforementioned fact that the convexity of
the relative entropy (more generally, the curvature-dimension condition)
works well for general Finsler manifolds.
Among other characterizations of lower Ricci curvature bounds for Riemannian manifolds,
we recently verified that the Bochner-Weitzenb\"ock formula makes sense
for general Finsler manifolds (\cite{OS2}).

The article is organized as follows.
After preliminaries for Minkowski and Finsler geometries,
we introduce the skew convexity in Section~\ref{sc:def},
and study the skew convexity of squared norms of Minkowski spaces in Section~\ref{sc:dist}.
In Section~\ref{sc:heat}, we discuss the heat flow on Minkowski spaces.
We give a detailed explanation on how to identify it with the gradient flow
of the relative entropy, because some results in \cite{OS} are not
directly applicable to noncompact spaces.
Section~\ref{sc:nonc} is devoted to a proof of Theorem~\ref{th:nonc}.
Finally, we consider the skew convexity of distance functions
on Finsler manifolds in Appendix.

\section{Preliminaries}\label{sc:prel}

We review the basics of Minkowski spaces and Finsler manifolds.
We refer to \cite{BCS} and \cite{Shlec} for Finsler geometry,
and to \cite[Chapter~14]{BCS} for Minkowski spaces.

\subsection{Minkowski spaces}\label{ssc:Mpre}

In this article, a \emph{Minkowski norm} will mean a nonnegative function
$\|\cdot\|:\R^n \lra [0,\infty)$ satisfying the following conditions.
\begin{enumerate}[(1)]
\item (\emph{Positive homogeneity})
$\|cx\|=c\|x\|$ holds for all $x \in \R^n$ and $c>0$.

\item (\emph{Strong convexity})
The function $\|\cdot\|^2/2$ is twice differentiable on $\R^n \setminus \{0\}$,
and the symmetric matrix
\begin{equation}\label{eq:g_ij}
\big( g_{ij}(x) \big)_{i,j=1}^n :=\bigg( \frac{1}{2}
 \frac{\del^2(\|\cdot\|^2)}{\del x^i \del x^j}(x) \bigg)_{i,j=1}^n
\end{equation}
is measurable in $x$ and uniformly elliptic in the sense that there are constants
$\lambda, \Lambda>0$ such that
\begin{equation}\label{eq:ell}
\lambda \sum_{i=1}^n (a^i)^2 \le \sum_{i,j=1}^n g_{ij}(x)a^ia^j
 \le \Lambda \sum_{i=1}^n (a^i)^2
\end{equation}
holds for all $x \in \R^n \setminus \{0\}$ and $(a^i) \in \R^n$
(in particular, $\|x\|>0$ for all $x \neq 0$).
\end{enumerate}

We call $(\R^n,\|\cdot\|)$ a \emph{Minkowski $($normed$)$ space}.
We remark that the strong convexity implies the strict convexity,
i.e., $\|x+y\|<\|x\|+\|y\|$ unless $x$ and $y$ are linearly dependent.
Note that the homogeneity is imposed only in positive direction,
so that $\|\!-x\| \neq \|x\|$ is allowed.
We also remark that the function $\|\cdot\|^2/2$ is twice differentiable
at the origin only in inner product spaces.
Given $x \in \R^n \setminus \{0\}$,
the matrix \eqref{eq:g_ij} defines the inner product $g_x$ of $\R^n$ by
\begin{equation}\label{eq:g_x}
g_x\big( (a^i),(b^j) \big) :=\sum_{i,j=1}^n g_{ij}(x) a^i b^j.
\end{equation}
This is the best approximation of the norm $\|\cdot\|$ in the direction $x$
in the sense that the unit sphere of $g_x$ is tangent
to that of $\|\cdot\|$ at $x/\|x\|$ up to the second order (Figure~1).
In particular, we have $g_x(x,x)=\|x\|^2$.
If the original norm comes from an inner product,
then $g_x$ coincides with it for all $x$.

\begin{center}
\begin{picture}(400,200)
\put(180,10){Figure~1}

\put(200,30){\vector(0,1){160}}
\put(80,110){\vector(1,0){240}}

\qbezier(250,110)(245,80)(210,70)
\qbezier(210,70)(180,65)(140,80)
\qbezier(140,80)(110,90)(110,110)
\qbezier(110,110)(110,145)(130,160)
\qbezier(130,160)(180,195)(230,150)
\qbezier(230,150)(247,135)(250,110)

\thicklines
\put(200,110){\ellipse{100}{50}}
\put(200,110){\vector(1,0){50}}

\put(205,115){$x/\|x\|$}
\put(140,139){$g_x(\cdot,\cdot)=1$}
\put(105,175){$\|\cdot\|=1$}

\end{picture}
\end{center}

We define the \emph{$2$-uniform convexity} and
\emph{smoothness constants} $\cC, \cS \in [1,\infty)$
as the least constants satisfying
\begin{align*}
\bigg\| \frac{x+y}{2} \bigg\|^2
&\le \frac{1}{2}\|x\|^2 +\frac{1}{2}\|y\|^2 -\frac{1}{4\cC^2}\|x-y\|^2, \\
\bigg\| \frac{x+y}{2} \bigg\|^2
&\ge \frac{1}{2}\|x\|^2 +\frac{1}{2}\|y\|^2 -\frac{\cS^2}{4}\|x-y\|^2
\end{align*}
for all $x,y \in \R^n$.
In other words, $\cC^{-2}$ and $\cS^2$ are the moduli of convexity
and concavity of $\|\cdot\|^2/2$, respectively.
Thanks to \eqref{eq:ell}, $\cC<\infty$ and $\cS<\infty$ hold.
Indeed, we know
\begin{equation}\label{eq:CS}
\cC=\sup_{x,y \in \R^n \setminus \{0\}} \frac{\|y\|}{g_x(y,y)^{1/2}}, \qquad
 \cS=\sup_{x,y \in \R^n \setminus \{0\}} \frac{g_x(y,y)^{1/2}}{\|y\|}
\end{equation}
(cf.\ \cite[Proposition~4.6]{Ouni}).
Note also that $\cC=1$ or $\cS=1$ holds if and only if the norm
is an inner product.

Denote by $\|\cdot\|_*$ the dual norm of $\|\cdot\|$.
Then the \emph{Legendre transform} $\cL:(\R^n,\|\cdot\|) \lra (\R^n,\|\cdot\|_*)$
associates $x$ with $\cL(x)$ satisfying $\|\cL(x)\|_*=\|x\|$
and $[\cL(x)](x)=\|x\|^2$.
Note that \eqref{eq:ell} ensures that $\cL(x)$ is indeed uniquely determined.
Moreover, $\cL(x)=(\cL_j(x))_{j=1}^n$ can be explicitly written as
\begin{equation}\label{eq:Leg}
\cL_j(x) =\frac{1}{2}\frac{\del(\|\cdot\|^2)}{\del x^j}(x) =\sum_{i=1}^n g_{ij}(x)x^i.
\end{equation}
The Legendre transform of inverse direction
$\cL^*:(\R^n,\|\cdot\|_*) \lra (\R^n,\|\cdot\|)$ is nothing but
the inverse map $\cL^*=\cL^{-1}$ by definition.
For a function $f:\R^n \lra \R$ and $x \in \R^n$ at where $f$ is differentiable,
we define the \emph{gradient vector} of $f$ at $x$ by
$\nabla f(x):=\cL^*(Df(x)) \in T_x\R^n$ (identified with $\R^n$).

\begin{remark}\label{rm:Mink}
We need the strong convexity of the norm to formulate and investigate the skew convexity
of functions as well as the heat equation, while the characterization of inner products
(Claim~\ref{cl:tri}) is valid among merely `convex' Minkowski norms
(i.e., its closed unit ball is a closed convex set containing the origin as an inner point).
In addition, the strict convexity will be a necessary condition when one studies
the contractivity of gradient flows (see Remark~\ref{rm:linf}).
\end{remark}

\subsection{Finsler manifolds}\label{ssc:Fpre}

Let $M$ be a connected $C^{\infty}$-manifold without boundary.
A nonnegative function $F:TM \lra [0,\infty)$ is called
a \emph{$C^{\infty}$-Finsler structure} if it is $C^{\infty}$
on $TM \setminus \{0\}$ ($\{0\}$ stands for the zero section)
and if $F|_{T_xM}$ is a Minkowski norm for all $x \in M$.
We call $(M,F)$ a $C^{\infty}$-Finsler manifold.
(We will consider only $C^{\infty}$-structures for simplicity.)

For each $v \in T_xM \setminus \{0\}$, we define the inner product
$g_v$ on $T_xM$ according to \eqref{eq:g_x}.
That is to say, given a local coordinate $(x^i)_{i=1}^n$
on an open set $U$ containing $x$, we consider the coordinate
of $T_xM$ as $v=\sum_{i=1}^n v^i (\del/\del x^i)|_x$ and define
\[ g_{ij}(v) :=\frac{1}{2} \frac{\del^2(F^2)}{\del v^i \del v^j}(v), \quad
 g_v\bigg( \sum_{i=1}^n a_i \frac{\del}{\del x^i}\Big|_x
 ,\sum_{j=1}^n b_j \frac{\del}{\del x^j}\Big|_x \bigg)
 :=\sum_{i,j=1}^n a_i b_j g_{ij}(v). \]
We denote by $\cC(x)$ and $\cS(x)$ the $2$-uniform convexity and
smoothness constants of $F|_{T_xM}$.
For a function $f:M \lra \R$ differentiable at $x \in M$,
define the gradient vector of $f$ at $x$ by $\nabla f(x):=\cL^*(Df(x))$
via the Legendre transform $\cL^*:T_x^*M \lra T_xM$.

The \emph{distance} from $x$ to $y$ is naturally defined as
$d(x,y):=\inf_{\gamma}\int_0^1 F(\dot{\gamma}) \,dt$, where
$\gamma:[0,1] \lra M$ runs over all differentiable curves from $x$ to $y$.
We remark that $d$ is nonsymmetric in general, namely $d(y,x) \neq d(x,y)$
may happen.
A \emph{geodesic} $\gamma:[0,l] \lra M$ is a locally length minimizing curve
of constant speed (i.e., $F(\dot{\gamma})$ is constant).
We say that $(M,F)$ is \emph{forward complete} if any geodesic
$\gamma:[0,l] \lra M$ is extended to a geodesic $\gamma:[0,\infty) \lra M$.
Then, for any $x,y \in M$, there is a minimal geodesic from $x$ to $y$
by the Hopf-Rinow theorem (\cite[Theorem~6.6.1]{BCS}).

Along a geodesic $\gamma:[0,l] \lra M$, $\gamma(s)$ with $s \in (0,l)$
is called a \emph{cut point} of $\gamma(0)$ if $\gamma|_{[0,s]}$ is minimal
and if $\gamma|_{[0,s+\ve]}$ is not minimal for any $\ve>0$.
Suppose that $\gamma(s)$ is not a cut point of $\gamma(0)$ for all $s \in (0,l]$,
and let $\xi$ and $\zeta$ be differentiable curves with
$\xi(0)=\gamma(0)$ and $\zeta(0)=\gamma(l)$.
Then we have the following first variation formula (\cite[Exercise~5.2.4]{BCS}):
\begin{equation}\label{eq:1vf}
\lim_{t \downarrow 0} \frac{d(\xi(t),\zeta(t))-d(\xi(0),\zeta(0))}{t}
 =\frac{g_{\dot{\gamma}(l)}(\dot{\gamma}(l),\dot{\zeta}(0))
 -g_{\dot{\gamma}(0)}(\dot{\gamma}(0),\dot{\xi}(0))}
 {l^{-1} \cdot d(\gamma(0),\gamma(l))}.
\end{equation}
As usual in discussing the contraction property,
this formula will play a vital role.

It is sometimes useful to consider the \emph{reverse} Finsler structure
$\overleftarrow{F}(v):=F(-v)$.
We will put an arrow $\leftarrow$ on those associated with $\overleftarrow{F}$,
for example, $\overleftarrow{d}\!(x,y)=d(y,x)$ and
$\overleftarrow{\nabla}f=-\nabla(-f)$.

\section{Skew convex functions}\label{sc:def}

We introduce the skew convexity of functions on a $C^{\infty}$-Finsler manifold
$(M,F)$, and will see that it is equivalent to the contractivity
of their gradient flows.
Although we shall work with $C^1$-functions for simplicity,
the same technique is also applicable to other classes of functions
(e.g., locally semi-convex functions, see Remark~\ref{rm:bicon} below).

Let us begin with the standard notion of convexity along geodesics.
A function $f:M \lra [-\infty,\infty]$ is said to be \emph{$K$-convex}
(or \emph{geodesically $K$-convex}) for $K \in \R$ if
\[ f\big( \gamma(t) \big) \le (1-t)f\big( \gamma(0) \big)
 +tf\big( \gamma(1) \big) -\frac{K}{2}(1-t)td\big( \gamma(0),\gamma(1) \big)^2 \]
holds for all geodesics $\gamma:[0,1] \lra M$ and $t \in [0,1]$.
If $f$ is $C^2$, then this is equivalent to
$\del^2[f \circ \gamma]/\del t^2 \ge Kd(\gamma(0),\gamma(1))^2$
and to
\[ \frac{\del}{\del t}\Big[ Df\big( \gamma(t) \big) \big( \dot{\gamma}(t) \big) \Big]
 \ge Kd\big( \gamma(0),\gamma(1) \big)^2. \]
Now, instead of $Df(\gamma)(\dot{\gamma})
 =-g_{\nabla(-f)(\gamma)}(\nabla(-f)(\gamma),\dot{\gamma})$
in the left hand side, we employ $-\cL(\dot{\gamma})(\nabla(-f)(\gamma))
 =-g_{\dot{\gamma}}(\nabla(-f)(\gamma),\dot{\gamma})$
for the skew convexity.

\begin{definition}[Skew convex functions]\label{df:bicon}
Let $f:M \lra \R$ be a $C^1$-function.
We say that $f$ is \emph{$K$-skew convex} for $K \in \R$ if,
for any pair of distinct points $x,y \in M$, there is a minimal geodesic
$\gamma:[0,1] \lra M$ from $x$ to $y$ such that
\begin{equation}\label{eq:bicon}
g_{\dot{\gamma}(1)}\big( \dot{\gamma}(1),\nabla(-f)(y) \big)
 -g_{\dot{\gamma}(0)}\big( \dot{\gamma}(0),\nabla(-f)(x) \big)
 \le -Kd(x,y)^2.
\end{equation}
\end{definition}

Recall that, on a Riemannian manifold $(M,g)$,
it holds $g_{\dot{\gamma}}=g$ and \eqref{eq:bicon} indeed implies
the $K$-convexity of $f$.
In the Finsler setting, however, $g_{\dot{\gamma}}$
is different from $g_{\nabla(-f)(\gamma)}$.

For a $C^1$-function $f:M \lra \R$ and any point $x \in M$,
there exists a $C^1$-curve $\xi:[0,\infty) \lra M$ satisfying
$\xi(0)=x$ and $\dot{\xi}(t)=\nabla(-f)(\xi(t))$ for all $t$.
We call such $\xi$ a \emph{gradient curve} of $f$.
For $K \in \R$, we say that
\emph{the gradient flow of $f$ is $K$-contractive} if
\[ d\big( \xi(t),\zeta(t) \big) \le e^{-Kt}d\big( \xi(0),\zeta(0) \big) \]
holds for all gradient curves $\xi,\zeta$ and $t \in [0,\infty)$.
Comparing \eqref{eq:bicon} with \eqref{eq:1vf},
we verify that the $K$-contractivity is equivalent to the $K$-skew convexity.

\begin{theorem}\label{th:equiv}
Let $(M,F)$ be a forward complete Finsler manifold, and let
$f:M \lra \R$ be a $C^1$-function.
Then the gradient flow of $f$ is $K$-contractive if and only if
$f$ is $K$-skew convex.
\end{theorem}

\begin{proof}
We first assume that $f$ is $K$-skew convex.
Fix two gradient curves $\xi,\zeta:[0,\infty) \lra M$ of $f$ and set
$l(t):=d(\xi(t),\zeta(t))$.
Given $t>0$, let $\gamma:[0,1] \lra M$ be a minimal geodesic
from $\xi(t)$ to $\zeta(t)$ such that \eqref{eq:bicon} holds.
Note that $\gamma(1/2)$ ($\zeta(t)$, resp.) is not a cut point of $\xi(t)$
($\gamma(1/2)$, resp.).
Thus the first variation formula \eqref{eq:1vf} shows that,
together with the triangle inequality,
\begin{align*}
\limsup_{\ve \downarrow 0} \frac{l(t+\ve)-l(t)}{\ve}
&\le \lim_{\ve \downarrow 0}
 \frac{d(\xi(t+\ve),\gamma(1/2))-d(\xi(t),\gamma(1/2))}{\ve} \\
&\quad +\lim_{\ve \downarrow 0}
 \frac{d(\gamma(1/2),\zeta(t+\ve))-d(\gamma(1/2),\zeta(t))}{\ve} \\
&= -g_{\dot{\gamma}(0)}\big( \dot{\gamma}(0)/l(t),\dot{\xi}(t) \big)
 +g_{\dot{\gamma}(1)}\big( \dot{\gamma}(1)/l(t),\dot{\zeta}(t) \big).
\end{align*}
By hypothesis, this yields $l'(t) \le -Kl(t)$ a.e.\ $t$.
Therefore we deduce from Gronwall's theorem that
$d(\xi(t),\zeta(t)) \le e^{-Kt}d(\xi(0),\zeta(0))$ holds.

To see the converse, suppose that the gradient flow of $f$ is $K$-contractive
and take a minimal geodesic $\gamma:[0,1] \lra M$.
Dividing $\gamma$ into $\gamma|_{[0,1/2]}$ and $\gamma|_{[1/2,1]}$
if necessary, we can assume that
$\gamma(s)$ is not a cut point of $\gamma(0)$ for all $s \in (0,1]$.
Consider gradient curves $\xi,\zeta:[0,\infty) \lra M$ of $f$ with
$\xi(0)=\gamma(0)$ and $\zeta(0)=\gamma(1)$,
and put $l(t):=d(\xi(t),\zeta(t))$ again.
Then it follows from the assumption that
\[ \frac{d}{dt}\Big|_{t=0+} \big[ e^{Kt}l(t) \big] \le 0. \]
This immediately implies the $K$-skew convexity,
as the first variation formula \eqref{eq:1vf} shows
\begin{align*}
\frac{d}{dt}\Big|_{t=0+}\big[ e^{Kt}l(t) \big]
&= Kl(0) +\frac{d}{dt}\Big|_{t=0+} l(t) \\
&= Kl(0) +g_{\dot{\gamma}(1)}\big( \dot{\gamma}(1)/l(0),\dot{\zeta}(0) \big)
 -g_{\dot{\gamma}(0)}\big( \dot{\gamma}(0)/l(0),\dot{\xi}(0) \big).
\end{align*}
$\qedd$
\end{proof}

\begin{remark}\label{rm:bicon}
We can replace the $C^1$-regularity in Definition~\ref{df:bicon}
with the local semi-convexity as follows (cf., e.g., \cite{Ly}, \cite{Ogra} for details).
We say that a function $f:M \lra \R$ is \emph{locally semi-convex}
if, for any $x \in M$, there are an open set $U \ni x$ and $K \in \R$
such that $f|_{U}$ is $K$-convex along any geodesic $\gamma:[0,1] \lra U$.
Define the \emph{local slope} of $f$ at $x \in M$ as
\begin{equation}\label{eq:grad}
|\grad f|(x):=\limsup_{y \to x} \frac{\max\{ f(x)-f(y),0 \}}{d(x,y)}.
\end{equation}
For each $x \in M$ with $|\grad f|(x)>0$,
there exists a unique unit vector $v \in T_xM$ satisfying
$\lim_{t \downarrow 0}\{f(\gamma(t))-f(x)\}/t=-|\grad f|(x)$,
where $\gamma$ is the geodesic with $\dot{\gamma}(0)=v$.
We define $\grad f(x):=|\grad f|(x) \cdot v$, and $\grad f(x):=0$ if $|\grad f|(x)=0$.
Then, from any initial point, there starts a gradient curve $\xi$ solving $\dot{\xi}(t)=\grad f(\xi(t))$ a.e.\ $t$.
The $K$-skew convexity can be defined by using $\grad f$ instead of $\nabla(-f)$ in \eqref{eq:bicon},
and the analogue of Theorem~\ref{th:equiv} holds by the same argument.
\end{remark}

\section{Skew convexity of squared norms}\label{sc:dist}

We study the skew convexity of the squared norm of a Minkowski space,
and compare it with the usual convexity along straight lines.
The more general case of distance functions on Finsler manifolds will be treated
in Appendix.

Let $(\R^n,\|\cdot\|)$ be a Minkowski space and set $f(x):=\|\!-x\|^2/2$.
Observe $\nabla(-f)(x)=-x$, so that the gradient curve $\xi$ of $f$
with $\xi(0)=x$ is given by $\xi(t)=e^{-t}x$.
Thus we see that the gradient flow of $f$ is $1$-contractive,
which shows that $f$ is $1$-skew convex.
This can be proved also by a direct calculation as, for any $v \in \R^n \setminus \{0\}$,
\[ g_v\big( v,\nabla(-f)(x+v) \big) -g_v\big( v,\nabla(-f)(x) \big)
 =-g_v(v,v)=-\|v\|^2. \]

\begin{remark}\label{rm:linf}
The above example obviously requires the strict convexity of the norm.
In fact, even the uniqueness of gradient flows fails in non-strictly convex normed spaces.
As pointed out by the referee, a typical example is the same function $f(x)=\|x\|_{\infty}^2/2$
in the $2$-dimensional $\ell_{\infty}$-space $(\R^2,\|\cdot\|_{\infty})$.
Any curve $\xi(t) =(e^{-t},h(t))$ satisfying $|h(t)| \le e^{-t}$ and $|h'(t)| \le e^{-t}$
is a gradient curve of $f$ in the metric sense of \cite{AGS}.
\end{remark}

In contrast to the $1$-skew convexity of $f$ above,
we can find a norm $\|\cdot\|$ of $\R^n$ such that
the function $\tilde{f}(x):=\langle x,x \rangle/2$ associated with
the Euclidean inner product $\langle \cdot,\cdot \rangle$ is not even $0$-skew convex,
although $\tilde{f}$ is convex along straight lines.
To see this, we observe $D\tilde{f}(x)=x$ and
$\nabla(-\tilde{f})(x)=\cL^*(-x)$ by identifying both $T_x\R^n$ and $T_x^*\R^n$
with $\R^n$.
Choosing $v=-x \neq 0$, we have
\[ g_v\big( v,\nabla(-\tilde{f})(0) \big) -g_v\big( v,\nabla(-\tilde{f})(-v) \big)
 =-g_v\big( v,\cL^*(v) \big) =-\langle \cL(v),\cL^*(v) \rangle, \]
where we used \eqref{eq:Leg} in the second equality.
However, $-\langle \cL(v),\cL^*(v) \rangle/\|v\|^2$ can be positive.
An example is illustrated in Figure~2, where we set $v=(0,1) \in \R^2$.
The parallelogram rounded to be strictly convex is the unit sphere of the norm $\|\cdot\|$.
Note that $\langle \cL(v),v \rangle =\|v\|^2=1$,
$\langle v,\cL^*(v) \rangle =\|\cL^*(v)\|^2$
and that $\langle \cL(v),\cL^*(v) \rangle <0$.

\begin{center}
\begin{picture}(400,210)
\put(180,10){Figure~2}

\put(200,30){\vector(0,1){170}}
\put(80,110){\vector(1,0){240}}

\put(130,165){\line(1,0){70}}
\put(200,170){\line(1,0){100}}
\qbezier(160,135)(200,165)(208,171)

\thicklines
\qbezier(200,165)(203,170)(290,170)
\qbezier(210,65)(270,125)(290,158)
\qbezier(290,170)(298,170)(290,158)
\qbezier(200,165)(199,164)(190,155)
\qbezier(200,55)(197,50)(110,50)
\qbezier(190,155)(130,95)(110,62)
\qbezier(110,50)(102,50)(110,62)
\qbezier(200,55)(201,56)(210,65)

\put(200,110){\vector(0,1){55}}
\put(200,110){\vector(3,2){110}}
\put(200,110){\vector(-3,4){41}}

\put(200,174){\line(1,0){4}}
\put(204,170){\line(0,1){4}}
\put(177.5,148.5){\line(-3,4){3}}
\put(174.5,152.5){\line(-4,-3){4}}

\put(220,60){$\|\cdot\|=1$}
\put(205,155){$v$}
\put(143,172){$\cL(v)$}
\put(270,182){$\cL^*(v)$}

\end{picture}
\end{center}

Given any $K<0$, by scaling $\tilde{f}$ (or the inner product) with sufficiently large $C>0$,
the convex function $C\tilde{f}$ is not $K$-skew convex.
This observation reveals that the skew convexity has no (obvious) relation with the usual convexity.
In addition, via Theorem~\ref{th:equiv}, we have seen that the usual convexity
does not imply the contractivity in non-Euclidean normed spaces.
This answers the question in \cite{AGS} quoted in the introduction.

\section{Heat flow on Minkowski spaces}\label{sc:heat}

In order to apply our technique to the heat flow on Minkowski spaces,
we regard it as the gradient flow of the relative entropy
with respect to the \emph{reverse} Wasserstein distance.
We refer to \cite{AGS} and \cite{Vi} for Wasserstein geometry
as well as the gradient flow theory.
Throughout the section, let $(\R^n,\|\cdot\|)$ be a Minkowski space
in the sense of Subsection~\ref{ssc:Mpre}.

\subsection{Wasserstein geometry over Minkowski spaces}

Let $\cP(\R^n)$ be the set of Borel probability measures on $\R^n$.
Define $\cP_2(\R^n) \subset \cP(\R^n)$ as the set of measures $\mu$
satisfying $\int_{\R^n} \|x\|^2 \,d\mu<\infty$
(note that then $\int_{\R^n} \|\!-x\|^2 \,d\mu<\infty$ holds as well).
The subset of absolutely continuous measures with respect to the Lebesgue measure $dx$
will be denoted by $\cP_2^{\ac}(\R^n) \subset \cP_2(\R^n)$.

Given $\mu, \nu \in \cP_2(\R^n)$, a probability measure
$\pi \in \cP(\R^n \times \R^n)$ is called a \emph{coupling} of $\mu$ and $\nu$
if $\pi(A \times \R^n)=\mu(A)$ and $\pi(\R^n \times A)=\nu(A)$
hold for all Borel sets $A \subset \R^n$.
We define the \emph{$L^2$-Wasserstein distance} from $\mu$ to $\nu$ by
\[ W_2(\mu,\nu):=\inf_{\pi} \bigg( \int_{\R^n \times \R^n}
 \|y-x\|^2 \,d\pi(x,y) \bigg)^{1/2}, \]
where $\pi$ runs over all couplings of $\mu$ and $\nu$.
A coupling $\pi$ attaining the infimum above is said to be \emph{optimal}.
We call $(\cP_2(\R^n),W_2)$ the \emph{$L^2$-Wasserstein space} over
$(\R^n,\|\cdot\|)$.

\begin{remark}\label{rm:W2}
(a) Thanks to \eqref{eq:ell}, our norm is comparable to an inner product.
In fact, \eqref{eq:CS} yields $\cC^{-1}\|y\| \le \sqrt{g_x(y,y)} \le \cS\|y\|$.
Then, if we denote by $W_2^{g_x}$ the Wasserstein distance
with respect to $g_x$, we have
$\cC^{-1}W_2(\mu,\nu) \le W_2^{g_x}(\mu,\nu) \le \cS W_2(\mu,\nu)$.
This relation is sometimes useful to apply known results in the Euclidean case.

(b) The least constant $c \ge 1$ satisfying $\|y\|^2 \le \langle y,y \rangle \le c\|y\|^2$
for some inner product $\langle \cdot,\cdot \rangle$ and all $y \in \R^n$
can not be bounded only by the dimension $n$,
unlike John's theorem for symmetric norms ($c \le n$).
For instance, consider the norm whose unit sphere is the standard unit sphere,
but with the center $(1-\ve,0,\ldots,0)$.
Letting $\ve \downarrow 0$, we have $c \to \infty$ (and $\cC, \cS \to \infty$).
\end{remark}

For $\mu \in \cP^{\ac}_2(\R^n)$ and $\nu \in \cP_2(\R^n)$,
there exists a semi-convex function $\varphi$
on an open set $\Omega \subset \R^n$ with $\mu(\Omega)=1$
such that $\pi:=(\id_{\R^n} \times T_1)_{\sharp}\mu$
provides the unique optimal coupling of $\mu$ and $\nu$,
where we set $T_t(x):=x+t\nabla\varphi(x)$ for $t \in [0,1]$
(by, e.g., \cite[Theorem~10.26]{Vi} under the conditions {\bf (locLip)}, {\bf (SC)}, {\bf (H$\infty$)}).
Moreover, $\mu_t:=(T_t)_{\sharp}\mu$ is the unique minimal geodesic
from $\mu_0=\mu$ to $\mu_1=\nu$.
Note that $\varphi$ is twice differentiable a.e.\ on $\Omega$
in the sense of Alexandrov, thus $T_t$ is well-defined and differentiable
a.e.\ on $\Omega$.

We introduce a Finsler structure of the Wasserstein space along the line of
\cite{Ot}, see \cite{OS} for more details in the case of compact Finsler manifolds.
We set
\[ \hat{T}\cP :=\{ \Phi=\nabla\varphi \,|\, \varphi \in C_c^{\infty}(\R^n) \} \]
and define the tangent space $(T_{\mu}\cP,F_{\mu})$ at $\mu \in \cP_2(\R^n)$
as the completion of $\hat{T}\cP$ with respect to the Minkowski norm
\[ F_{\mu}(\Phi):=\bigg( \int_{\R^n} \|\Phi\|^2 \,d\mu \bigg)^{1/2} \]
(of the space of measurable vector fields $\Phi$ with $F_{\mu}(\Phi)<\infty$).
Similarly, the cotangent space $(T^*_{\mu}\cP,F^*_{\mu})$
is defined as the completion of
$\hat{T}^*\cP :=\{ \alpha=D\varphi \,|\, \varphi \in C_c^{\infty}(\R^n) \}$
with respect to
$F^*_{\mu}(\alpha):=(\int_{\R^n} \|\alpha\|_*^2 \,d\mu)^{1/2}$.
We define the Legendre transform $\cL^*_{\mu}:T_{\mu}^*\cP \lra T_{\mu}\cP$
in the pointwise way that $\cL^*_{\mu}(D\varphi)=\nabla\varphi$.

We say that a curve $(\mu_t)_{t \in I} \subset \cP_2(\R^n)$
on an open interval $I \subset \R$ is ($2$-)\emph{absolutely continuous}
if there is some function $h \in L^2(I)$ such that
\[ W_2(\mu_t,\mu_{\tau}) \le \int_t^{\tau} h(r) \,dr \]
holds for all $t,\tau \in I$ with $t<\tau$.
Note that an absolutely continuous curve is continuous.
Given an absolutely continuous curve $(\mu_t)_{t \in I}$,
the \emph{forward absolute gradient} (or the \emph{metric speed})
\[ |\dot{\mu}_t|:=\lim_{\tau \to t}
 \frac{W_2(\mu_{\min\{\tau,t\}},\mu_{\max\{\tau,t\}})}{|\tau-t|} \]
is well-defined for a.e.\ $t \in I$
(\cite[Theorem~1.1.2]{AGS}, \cite[Lemma~7.1]{OS}).
We can associate $(\mu_t)_{t \in I}$ with a Borel vector field
$\Phi$ on $I \times \R^n$ (with $\Phi_t(x):=\Phi(t,x) \in T_x\R^n$) satisfying
\begin{itemize}
\item $\Phi_t \in T_{\mu_t}\cP$ at a.e.\ $t \in I$,
\item the \emph{continuity equation} $\del_t \mu_t+\div(\Phi_t \mu_t)=0$
in the weak sense that
\begin{equation}\label{eq:conti}
\int_I \int_{\R^n} \{ \del_t \psi+D\psi(\Phi) \} \,d\mu_t dt=0
\end{equation}
for all $\psi \in C_c^{\infty}(I \times \R^n)$
(\cite[Theorem~8.3.1]{AGS}, \cite[Theorem~7.3]{OS}).
\end{itemize}
Such a vector field $\Phi$ is unique up to a difference on a null measure set with respect to $d\mu_t dt$,
and we have $F_{\mu_t}(\Phi_t)=|\dot{\mu}_t|$ a.e.\ $t \in I$.
We will call $\Phi$ the \emph{tangent vector field} of  the curve
$(\mu_t)_{t \in I}$ and write $\dot{\mu}_t=\Phi_t$.

Now, consider a function $Q:\cP_2(\R^n) \lra [-\infty,\infty]$.
We say that $Q$ is \emph{differentiable} at $\mu \in \cP_2(\R^n)$
if $-\infty<Q(\mu)<\infty$ and if there is some $\alpha \in T_{\mu}^*\cP$ such that
\begin{equation}\label{eq:DQ}
\int_{\R^n} \alpha(\Phi) \,d\mu \ge \limsup_{t \to 0}
 \frac{Q(\mu_t)-Q(\mu)}{t}
\end{equation}
holds for every minimal geodesic $(\mu_t)_{t \in [0,1]}$ with
$\mu_t=(T_t)_{\sharp}\mu$ and $T_t(x)=x+t\Phi(x)$,
and that equality holds in \eqref{eq:DQ} if $\Phi \in \hat{T}\cP$
(with $\lim_{t \to 0}$ in place of $\limsup_{t \to 0}$).
Such a one-form $\alpha$ is unique in $T^*_{\mu}\cP$ up to a difference on a
$\mu$-null measure set.
Thus we write $DQ(\mu)=\alpha$ and define the \emph{gradient vector}
of $Q$ at $\mu$ by $\nabla_W Q(\mu):=\cL_{\mu}^*(DQ(\mu))$.

\begin{definition}[Gradient curves in $(\cP_2(\R^n),W_2)$]\label{df:Wgf}
We say that an absolutely continuous curve
$(\mu_t)_{t \in [0,T)} \subset \cP_2(\R^n)$
with $T \in (0,\infty]$ is a \emph{gradient curve} of $Q$
if $\dot{\mu}_t=\nabla_W (-Q)(\mu_t)$ holds for a.e.\ $t \in (0,T)$.
\end{definition}

We remark that the differentiability of $-Q$ at a.e.\ $t \in (0,T)$ is included in the above definition.

\subsection{Nonlinear heat equation and global solutions}

We define the (distributional) \emph{Finsler Laplacian} $\Delta$
acting on $u \in H^1_{\loc}(\R^n)$ by
\[ \int_{\R^n} \psi \Delta u \,dx =-\int_{\R^n} D\psi(\nabla u) \,dx \]
for all $\psi \in C^{\infty}_c(\R^n)$.
Note that $\Delta$ is a nonlinear operator since the Legendre transform is nonlinear
(unless the norm $\|\cdot\|$ comes from an inner product).
We consider the associated \emph{heat equation} $\del_t u=\Delta u$
also in the weak form.
We have seen in \cite[Theorem~3.4]{OS} that, given $u_0 \in H^1_0(\R^n)$
and $T>0$, there exists a unique \emph{global solution}
$u \in L^2([0,T],H^1_0(\R^n)) \cap H^1([0,T],L^2(M))$
to $\del_t u=\Delta u$ in the weak sense that
\begin{equation}\label{eq:heat}
\int_{\R^n} \psi \del_t u \,dx=-\int_{\R^n} D\psi(\nabla u) \,dx
\end{equation}
holds for all $t \in [0,T]$ and $\psi \in H^1_0(\R^n)$.

We can also regard $(u_t)_{t \in [0,T]}$ as a weak solution to the heat equation
$\del_t v=\Delta\!^{\nabla u}v$ associated with the linear, second order,
time-dependent differential operator
\begin{equation}\label{eq:Lapv}
\Delta\!^{\nabla u}v:=\div\bigg( \sum_{i,j=1}^n
 g^{ij}(\nabla u) \frac{\del v}{\del x^i} \frac{\del}{\del x^j} \bigg),
\end{equation}
where $(g^{ij})$ stands for the inverse matrix of $(g_{ij})$ and
$\nabla u(x)$ is replaced with some nonzero vector if $\nabla u(x)=0$
(in a measurable way).
By virtue of \eqref{eq:ell}, $(g^{ij}(\nabla u))$ is globally uniformly elliptic
with respect to the Euclidean inner product.
Therefore the classical theory due to Nash~\cite{Na}, Moser~\cite{Mo},
Aronson~\cite{Ar} and others yields the parabolic Harnack inequality
as well as the Gaussian estimates from both sides for fundamental solutions
(see also \cite{Sal} for the Riemannian case).
Moreover, the continuous version of $u$ is $H^2_{\loc}$ in $x$ and
$C^{1,\alpha}$ on $(0,\infty) \times \R^n$ (\cite[Theorems~4.6, 4.9]{OS}).

The following lemma allows us to consider $(u_t\,dx)_{t \ge 0}$ as a curve in
$\cP_2(\R^n)$.

\begin{lemma}\label{lm:heat}
Let $(u_t)_{t \ge 0} \subset H^1_0(\R^n)$ be a global solution to the heat equation.
Then we have the following.
\begin{enumerate}[{\rm (i)}]
\item {\rm (Mass preserving)}
If $u_0 \,dx \in \cP(\R^n)$, then $u_t \,dx \in \cP(\R^n)$ for all $t>0$.

\item
If $u_0 \,dx \in \cP_2(\R^n)$, then $u_t \,dx \in \cP_2(\R^n)$
for all $t >0$.

\item {\rm (Continuity in $W_2$)}
If $u_0 \,dx \in \cP_2(\R^n)$, then $\lim_{t \to 0}W_2(u_0 \,dx,u_t \,dx)=0$.
\end{enumerate}
\end{lemma}

\begin{proof}
(i) This easily follows from the existence of the fundamental solution
$q^u$ to the equation $\del_t v=\Delta\!^{\nabla u}v$.
Precisely, $u_t(x)=\int_{\R^n} q^u(t,x;0,y) u_0(y) \,dy$ and
$\int_{\R^n} q^u(t,x;0,y) \,dx=1$ imply $\int_{\R^n} u_t \,dx=1$.

(ii) By virtue of the upper Gaussian bound (cf.\ \cite[Corollary~6.2]{Sal}),
we have
\[ q^u(t,x;0,y) \le C_1t^{-n/2} \exp\bigg(\! -\frac{|x-y|^2}{C_2t} \bigg), \]
where $|\cdot|$ stands for the Euclidean norm and $C_1$, $C_2$ depend only on $\|\cdot\|$.
Thus we obtain
\begin{align*}
&\int_{\R^n} \|x\|^2 u_t(x) \,dx
 \le \int_{\R^n} \int_{\R^n} 2(\|x-y\|^2 +\|y\|^2) q^u(t,x;0,y) u_0(y) \,dydx \\
&\le 2C_1t^{-n/2} \int_{\R^n} \int_{\R^n}
 \|x-y\|^2 \exp\bigg(\! -\frac{|x-y|^2}{C_2t} \bigg) u_0(y) \,dxdy
 +2\int_{\R^n} \|y\|^2 u_0(y) \,dy \\
&\le C_3 t +2\int_{\R^n} \|y\|^2 u_0(y) \,dy <\infty,
\end{align*}
where we used the fact $se^{-s} \le e^{-s/2}$ in the last inequality
and $C_3$ depends only on $C_1,C_2$, the norm $\|\cdot\|$ and $n$.

(iii)
It suffices to show that $u_t \,dx$ weakly converges to $u_0 \,dx$ and
\[ \lim_{R \to \infty} \limsup_{t \to 0} \int_{\{|x| \ge R\}} |x|^2 u_t(x) \,dx=0 \]
by Remark~\ref{rm:W2}(a) (cf.\ \cite[Theorem~6.9]{Vi}).
For the weak convergence, thanks to (i) and \cite[Remark~5.1.6]{AGS},
it is sufficient to show the convergence for test functions $f \in C_c^{\infty}(\R^n)$.
This immediately follows from \eqref{eq:heat}, indeed,
the Cauchy-Schwarz inequality yields
\[ \bigg| \int_0^{\tau} \int_{\R^n} f \del_t u \,dxdt \bigg|
 \le \bigg( \tau \int_{\R^n} \| Df \|_*^2 \,dx \bigg)^{1/2}
 \bigg( \int_0^{\tau} \int_{\R^n} \| \nabla u_t \|^2 \,dxdt \bigg)^{1/2}
 \to 0 \]
as $\tau$ tends to zero.
The latter condition can be seen similarly to (ii).
As $|x|<2|x-y|$ if $|x| \ge R$ and $|y|<R/2$, we have
\begin{align*}
\int_{\{|x| \ge R\}} |x|^2 u_t(x) \,dx
&\le \int_{\{|x| \ge R\}} \int_{\{|y| \ge R/2\}} 2(|x-y|^2 +|y|^2) q^u(t,x;0,y) u_0(y) \,dydx \\
&\quad +\int_{\{|x| \ge R\}} \int_{\{|y|<R/2\}} 4|x-y|^2 q^u(t,x;0,y) u_0(y) \,dydx \\
&\le Ct +2\int_{\{|y| \ge R/2\}} |y|^2 u_0(y) \,dy \to 0
\end{align*}
as $t \to 0$ and then $R \to \infty$.
$\qedd$
\end{proof}

\subsection{Relative entropy and heat flow as its gradient flow}

We define the \emph{relative entropy} (with respect to the Lebesgue measure) by
\[ \Ent(\mu):=\int_{\R^n} \rho \log\rho \,dx \in (-\infty,\infty] \]
for $\mu=\rho\, dx \in \cP_2^{\ac}(\R^n)$, and $\Ent(\mu):=\infty$
for $\mu \in \cP_2(\R^n) \setminus \cP_2^{\ac}(\R^n)$.
See \cite[Section~4]{JKO} (and Remark~\ref{rm:W2}(a))
for the fact $\Ent >-\infty$ on $\cP_2^{\ac}(\R^n)$.
We know that $\Ent$ is convex along geodesics in $(\cP_2(\R^n),W_2)$
(\cite[page~908]{Vi}).
There is a well established theory on gradient flows of such convex functionals,
for which we refer to \cite{AGS}.
Here we explain that a global solution to the heat equation gives
the gradient flow of the relative entropy along the lines of \cite{OS}, \cite{AGS} and \cite{Er}.
The next lemma corresponds to \cite[Proposition~7.7]{OS},
see also \cite[Theorem~10.4.17]{AGS} and \cite[Proposition~4.3]{Er}.

\begin{lemma}\label{lm:DEnt}
For $\mu=\rho \,dx \in \cP_2^{\ac}(\R^n)$ with $\rho \in H^1_0(\R^n)$ and $\Ent(\mu)<\infty$,
the following are equivalent.
\begin{enumerate}[{\rm (I)}]
\item $-\Ent$ is differentiable at $\mu$.
\item $\|\nabla(-\rho)\| /\rho \in L^2(\R^n,\mu)$.
\end{enumerate}
Moreover, then $\nabla(-\rho)/\rho \in T_{\mu}\cP$ and we have
\[ \nabla_W(-\Ent)(\mu) =\frac{\nabla(-\rho)}{\rho}. \]
\end{lemma}

\begin{proof}
(I) $\Rightarrow$ (II):
It suffices to see that $|\grad \Ent|(\mu)$ defined as in \eqref{eq:grad} is finite,
then \cite[Proposition~4.3]{Er} yields (II)
because our norm is comparable to a Euclidean norm.
Take $\nu \in \cP_2^{\ac}(\R^n)$ with $\Ent(\nu)<\infty$ and let
$\mu_t=(T_t)_{\sharp}\mu$ with $T_t(x)=x+t\Phi(x)$ be the minimal geodesic
from $\mu_0=\mu$ to $\mu_1=\nu$.
We deduce from the convexity of $\Ent(\mu_t)$ that
\[ \limsup_{t \to 0}\frac{\Ent(\mu)-\Ent(\mu_t)}{W_2(\mu,\mu_t)}
 \ge \frac{\Ent(\mu)-\Ent(\nu)}{W_2(\mu,\nu)}. \]
Thanks to the differentiability of $-\Ent$, we observe
\[ \limsup_{t \to 0}\frac{\Ent(\mu)-\Ent(\mu_t)}{t}
 \le \int_{\R^n} [D(-\Ent)(\mu)](\Phi) \,d\mu
 \le F_{\mu}\big( \nabla_W(-\Ent)(\mu) \big) \cdot W_2(\mu,\nu). \]
Therefore we have $|\grad \Ent|(\mu) \le F_{\mu}(\nabla_W(-\Ent)(\mu))<\infty$.

(II) $\Rightarrow$ (I):
Again due to \cite[Proposition~4.3]{Er}, we obtain $|\grad \Ent|(\mu)<\infty$
as well as $D(-\rho)/\rho \in T^*_{\mu}\cP$.
For any minimal geodesic $(\mu_t)_{t \in [0,1]}$ as in (I) $\Rightarrow$ (II),
\cite[Theorem~23.14]{Vi} (see also \cite[Proposition~7.7]{OS} for the compact case) shows that
\[ \lim_{t \to 0}\frac{\Ent(\mu)-\Ent(\mu_t)}{t}
 =\int_{\R^n} \div(\Phi) \,d\mu
 \le \int_{\R^n} \bigg[ \frac{D(-\rho)}{\rho} \bigg](\Phi) \,d\mu, \]
and equality holds if $\Phi \in \hat{T}\cP$.
Thus $-\Ent$ is differentiable and it holds $\nabla_W(-\Ent)(\mu)=\nabla(-\rho)/\rho$.
$\qedd$
\end{proof}

The following theorem is a slight modification of
\cite[Theorem~7.8]{OS} adapted to noncompact spaces.
For the sake of simplicity, we are concerned with the reverse heat equation,
that is the heat equation with respect to the reverse norm
$\|x\|_{\leftarrow} :=\|\!-x\|$.
Since $\overleftarrow{\nabla}u=-\nabla(-u)$, we can write it as
\begin{equation}\label{eq:rheat}
\int_{\R^n} \psi \del_t u \,dx
 =-\int_{\R^n} D\psi(\overleftarrow{\nabla}u) \,dx
 =\int_{\R^n} D\psi\big( \nabla(-u) \big) \,dx.
\end{equation}

\begin{theorem}[Heat flow as gradient flow]\label{th:gf=hf}
\begin{enumerate}[{\rm (i)}]
\item
Let $(\rho_t)_{t \ge 0} \subset H^1_0(\R^n)$ be a global solution to the reverse
heat equation with $\rho_0 \,dx \in \cP^{\ac}_2(\R^n)$.
Then $\mu_t:=\rho_t \,dx$ is a gradient curve of the relative entropy
$($in the sense of Definition~$\ref{df:Wgf})$.

\item
Conversely, let $(\mu_t)_{t \ge 0} \subset \cP_2^{\ac}(\R^n)$
be a gradient curve of the relative entropy, put $\mu_t=\rho_t \,dx$
and assume that $\rho_t \in H^1_0(\R^n)$ for a.e.\ $t$.
Then $\rho_t$ is a global solution to the reverse heat equation.
\end{enumerate}
\end{theorem}

\begin{proof}
(i) We first of all remark that $\Ent(\mu_t)<\infty$ for all $t>0$ by the upper Gaussian estimate
for fundamental solutions (see the proof of Lemma~\ref{lm:heat}(ii)).
It follows from the reverse heat equation \eqref{eq:rheat} that
$\nabla(-\rho)/\rho$ satisfies the continuity equation \eqref{eq:conti}
along the curve $(\mu_t)_{t \ge 0}$.
More generally, we have
\begin{equation}\label{eq:heat-c}
\int_{\R^n} \psi_T \,d\mu_T -\int_{\R^n} \psi_{\tau} \,d\mu_{\tau}
= \int_{\tau}^T \int_{\R^n}
 \bigg\{ \del_t\psi +D\psi\bigg( \frac{\nabla(-\rho_t)}{\rho_t} \bigg) \bigg\} \,d\mu_t dt
\end{equation}
for all $\psi \in C_c^{\infty}([0,\infty) \times \R^n)$ and $0<\tau<T<\infty$.
Then we obtain, by choosing a test function $\psi$
approximating $\log(\max\{\rho_t,\ve\})-\log\ve$ on $[\tau,T] \times \R^n$
and then letting $\ve \downarrow 0$,
\[ \int_{\tau}^T \int_{\R^n} \frac{\|\nabla(-\rho_t)\|^2}{\rho_t^2} \,d\mu_t dt
 =\Ent(\mu_{\tau})-\Ent(\mu_T) <\infty. \]
Hence $(\mu_t)_{t>0}$ is absolutely continuous (see \cite[Theorem~8.3.1]{AGS}) and,
as $T>0$ was arbitrary, we find $\|\nabla(-\rho_t)\|/\rho_t \in L^2(\R^n,\mu_t)$ at a.e.\ $t$.
This implies that $-\Ent$ is differentiable at $\mu_t$ and
$\nabla_W(-\Ent)(\mu_t)=\nabla(-\rho_t)/\rho_t \in T_{\mu_t}\cP$
a.e.\ $t$ (Lemma~\ref{lm:DEnt}).
Combining $\nabla(-\rho_t)/\rho_t \in T_{\mu_t}\cP$ with \eqref{eq:heat-c}, we conclude
$\dot{\mu_t}=\nabla(-\rho_t)/\rho_t$ and thus $(\mu_t)_{t \ge 0}$
is a gradient curve of $\Ent$ in the sense of Definition~\ref{df:Wgf}.

(ii) Note that Lemma~\ref{lm:DEnt} ensures
$\dot{\mu}_t=\nabla_W(-\Ent)(\mu_t) =\nabla(-\rho_t)/\rho_t$ for a.e.\ $t$.
Then the continuity equation with $\Phi=\nabla(-\rho)/\rho$
immediately implies the reverse heat equation.
$\qedd$
\end{proof}

\begin{remark}\label{rm:DEnt}
The formula $\nabla_W(-\Ent)(\mu)=\nabla(-\rho)/\rho$
in Lemma~\ref{lm:DEnt} has an extra importance in the Finsler/Minkowski setting.
The reverse heat equation \eqref{eq:rheat} is rewritten via the integration by parts
as $\int_{\R^n} \psi \del_t \rho \,dx
 =\int_{\R^n} \Delta\!^{\nabla(-\rho)}\psi \,\rho \,dx$
(in other words, $\overleftarrow{\Delta}\rho=\Delta^{\nabla(-\rho)}\rho$, see \eqref{eq:Lapv}).
Then the homogeneity $g_{\nabla(-\rho)}=g_{\nabla(-\rho)/\rho}$ guarantees
that a \emph{formal} calculation with respect to the time-dependent Riemannian
structure $g_{\nabla(-\rho)}$ verifies Theorem~\ref{th:gf=hf}.
\end{remark}

\subsection{Skew convexity and Wasserstein contraction}

To show an analogous result to Theorem~\ref{th:equiv} for the relative entropy,
we prove the first variation formula for the Wasserstein distance along heat flow
(along the line of \cite[Section~10.2]{AGS}).

\begin{proposition}[First variation formula for $W_2$ along heat flow]\label{pr:W1vf}
For any global solutions
$(\rho_t)_{t \ge 0}, (\sigma_t)_{t \ge 0} \subset H^1_0(\R^n)$
to the reverse heat equation \eqref{eq:rheat} such that
$\mu_t=\rho_t \,dx \in \cP_2^{\ac}(\R^n)$ and
$\nu_t=\sigma_t \,dx \in \cP_2^{\ac}(\R^n)$, we have
\begin{equation}\label{eq:W1vf}
\lim_{\tau \downarrow t}
 \frac{W_2(\mu_{\tau},\nu_{\tau})^2-W_2(\mu_t,\nu_t)^2}{2(\tau-t)}
 =\int_{\R^n} g_{\dot{\omega}_1}(\dot{\omega}_1,\dot{\nu}_t) \,d\nu_t
 -\int_{\R^n} g_{\dot{\omega}_0}(\dot{\omega}_0,\dot{\mu}_t) \,d\mu_t
\end{equation}
for all $t>0$, where $\omega:[0,1] \lra \cP_2(\R^n)$ is the minimal
geodesic from $\mu_t$ to $\nu_t$.
\end{proposition}

\begin{proof}
Set $l(t):=W_2(\mu_t,\nu_t)$,
fix $\delta>0$ and define $\widetilde{\Gamma}$ as the set of
continuous curves $\xi:[t,t+\delta] \lra \R^n$ endowed with the uniform topology.
For $\tau \in [t,t+\delta]$, we define the evaluation map
$e_{\tau}:\widetilde{\Gamma} \lra \R^n$ by $e_{\tau}(\xi):=\xi(\tau)$.
Then there exist probability measures $\Pi, \Xi \in \cP(\widetilde{\Gamma})$ such that
$(e_{\tau})_{\sharp}\Pi=\mu_{\tau}$, $(e_{\tau})_{\sharp}\Xi=\nu_{\tau}$
for all $\tau \in [t,t+\delta]$ and that $\Pi$, $\Xi$ are concentrated
on the set of $C^1$-curves $\xi$, $\zeta$ solving
\[ \dot{\xi}(\tau)=\frac{\nabla(-\rho_{\tau})}{\rho_{\tau}} \big( \xi(\tau) \big), \qquad
 \dot{\zeta}(\tau)=\frac{\nabla(-\sigma_{\tau})}{\sigma_{\tau}}\big( \zeta(\tau) \big) \]
for all $\tau \in [t,t+\delta]$, respectively (cf.\ \cite[Theorem~8.2.1]{AGS}).
We remark that $\rho_{\tau},\sigma_{\tau}>0$ for $\tau>0$ by the lower Gaussian estimate.

To see `$\le$' of \eqref{eq:W1vf},
we disintegrate $\Pi$ and $\Xi$ by using $\mu_t$ and $\nu_t$ as
$d\Pi=d\Pi_x^t d\mu_t(x)$ and $d\Xi=d\Xi_y^t d\nu_t(y)$,
where $\Pi_x^t, \Xi_y^t \in \cP(\widetilde{\Gamma})$ concentrate on the sets
$e_t^{-1}(x)$ and $e_t^{-1}(y)$, respectively.
Take the unique minimal geodesic $\omega:[0,1] \lra \cP_2(\R^n)$ from
$\mu_t$ to $\nu_t$, and let $\pi_t$ be the unique optimal coupling
of $\mu_t$ and $\nu_t$.
Then we find, for each $\tau \in [t,t+\delta]$,
\[ l(\tau)^2 \le
 \int_{\R^n \times \R^n} \int_{\widetilde{\Gamma} \times \widetilde{\Gamma}}
 \|\zeta(\tau)-\xi(\tau)\|^2 \,d\Pi_x^t(\xi) d\Xi_y^t(\zeta) d\pi_t(x,y). \]
We deduce from the first variation formula \eqref{eq:1vf} on the underlying space
$(\R^n,\|\cdot\|)$ that
\begin{align*}
&\lim_{\tau \downarrow t}
 \frac{\|\zeta(\tau)-\xi(\tau)\|^2-\|\zeta(t)-\xi(t)\|^2}{2(\tau -t)} \\
&= g_{\zeta(t)-\xi(t)}\big( \zeta(t)-\xi(t),\dot{\zeta}(t) \big)
 -g_{\zeta(t)-\xi(t)}\big( \zeta(t)-\xi(t),\dot{\xi}(t) \big) \\
&= g_{\zeta(t)-\xi(t)}\bigg( \zeta(t)-\xi(t),
 \frac{\nabla(-\sigma_t)}{\sigma_t}\big( \zeta(t) \big)
 -\frac{\nabla(-\rho_t)}{\rho_t}\big( \xi(t) \big) \bigg)
\end{align*}
for $\Pi$-a.e.\ $\xi$ and $\Xi$-a.e.\ $\zeta$.
Since $\rho$ and $\sigma$ are $C^1$ on $(0,\infty) \times \R^n$,
this convergence is uniform on
\[ \Omega_{\ve}:=\{ x \in \R^n \,|\,
 \|x\|<\ve^{-1},\ \rho_t(x)>\ve,\ \sigma_t(x)>\ve \} \]
for each $\ve>0$.
In order to see that the effect of $\R^n \setminus \Omega_{\ve}$
is negligible as $\ve$ tends to zero, we observe from
\begin{align*}
&\|\zeta(\tau)-\xi(\tau)\|^2-\|\zeta(t)-\xi(t)\|^2 \\
&\le (\|\zeta(\tau)-\xi(\tau)\|+\|\zeta(t)-\xi(t)\|)
 (\|\zeta(\tau)-\zeta(t)\|+\|\xi(t)-\xi(\tau)\|)
\end{align*}
that
\begin{align*}
&\int_{\R^n \times \R^n} \int_{\widetilde{\Gamma} \times \widetilde{\Gamma}}
 \frac{\|\zeta(\tau)-\xi(\tau)\|^2-\|\zeta(t)-\xi(t)\|^2}{2(\tau -t)}
 \,d\Pi_x^t(\xi) d\Xi_y^t(\zeta) d\pi_t(x,y) \\
&\le \bigg( \int_{\R^n \times \R^n}
 \int_{\widetilde{\Gamma} \times \widetilde{\Gamma}}
 \{ \|\zeta(\tau)-\xi(\tau)\|^2+\|\zeta(t)-\xi(t)\|^2 \}
 \,d\Pi_x^t(\xi) d\Xi_y^t(\zeta) d\pi_t(x,y) \bigg)^{1/2} \\
&\qquad \times \bigg(
 \int_{\widetilde{\Gamma}} \frac{\|\zeta(\tau)-\zeta(t)\|^2}{(\tau -t)^2} \,d\Xi(\zeta)
 +\int_{\widetilde{\Gamma}} \frac{\|\xi(t)-\xi(\tau)\|^2}{(\tau -t)^2} \,d\Pi(\xi) \bigg)^{1/2}.
\end{align*}
This is finite uniformly in $\tau \in (t,t+\delta]$, because
\begin{align*}
&\frac{1}{(\tau -t)^2} \int_{\widetilde{\Gamma}} \|\zeta(\tau)-\zeta(t)\|^2 \,d\Xi(\zeta)
 \le \frac{1}{(\tau -t)^2} \int_{\widetilde{\Gamma}}
 \bigg( \int_t^{\tau} \|\dot{\zeta}(s)\| \,ds \bigg)^2 d\Xi(\zeta) \\
&\le \frac{1}{\tau -t} \int_{\widetilde{\Gamma}}
 \int_t^{\tau} \|\dot{\zeta}(s)\|^2 \,dsd\Xi(\zeta)
 =\frac{1}{\tau -t} \int_t^{\tau} |\dot{\nu}_s|^2 \,ds.
\end{align*}
Therefore we obtain
\[ \limsup_{\tau \downarrow t} \frac{l(\tau)^2-l(t)^2}{2(\tau-t)}
 \le \int_{\R^n \times \R^n} \int_{\widetilde{\Gamma} \times \widetilde{\Gamma}}
 g_{y-x}\big( y-x,\dot{\zeta}(t)-\dot{\xi}(t) \big)
 \,d\Pi_x^t(\xi) d\Xi_y^t(\zeta) d\pi_t(x,y). \]
Note that
\begin{align*}
\int_{\R^n \times \R^n} \int_{\widetilde{\Gamma}}
 g_{y-x}\big( y-x,\dot{\xi}(t) \big) \,d\Pi_x^t(\xi)d\pi_t(x,y)
&=\int_{\R^n} \int_{\widetilde{\Gamma}}
 g_{\dot{\omega}_0(x)}\big( \dot{\omega}_0(x),\dot{\xi}(t) \big)
 \,d\Pi_x^t(\xi)d\mu_t(x) \\
&= \int_{\R^n} g_{\dot{\omega}_0}(\dot{\omega}_0,\dot{\mu}_t) \,d\mu_t.
\end{align*}
Hence we have
\[ \limsup_{\tau \downarrow t} \frac{l(\tau)^2-l(t)^2}{2(\tau-t)}
 \le \int_{\R^n} g_{\dot{\omega}_1}(\dot{\omega}_1,\dot{\nu}_t) \,d\nu_t
 -\int_{\R^n} g_{\dot{\omega}_0}(\dot{\omega}_0,\dot{\mu}_t) \,d\mu_t. \]

To see the reverse inequality, we fix $\tau \in (t,t+\delta)$,
take the optimal coupling $\pi_{\tau}$ of $\mu_{\tau}$ and $\nu_{\tau}$,
and disintegrate $\Pi$ and $\Xi$ as $d\Pi=d\Pi_x^{\tau} d\mu_{\tau}(x)$
and $d\Xi=d\Xi_y^{\tau} d\nu_{\tau}(y)$.
Observe that
\[ l(\tau)^2-l(t)^2 \ge
 \int_{\R^n \times \R^n} \int_{\widetilde{\Gamma} \times \widetilde{\Gamma}}
 \{ \|\zeta(\tau)-\xi(\tau)\|^2-\|\zeta(t)-\xi(t)\|^2 \}
 \,d\Pi_x^{\tau}(\xi) d\Xi_y^{\tau}(\zeta) d\pi_{\tau}(x,y). \]
Since the function
\[ [0,1] \ni s \longmapsto
 \| \{(1-s)\zeta(t)+s\zeta(\tau)\}-\{(1-s)\xi(t)+s\xi(\tau)\} \|^2 \]
is convex, the first variation formula \eqref{eq:1vf} (at $s=0$) yields that
\[  \|\zeta(\tau)-\xi(\tau)\|^2-\|\zeta(t)-\xi(t)\|^2
 \ge 2g_{\zeta(t)-\xi(t)}\big( \zeta(t)-\xi(t),
 \{ \zeta(\tau)-\zeta(t) \}-\{ \xi(\tau)-\xi(t) \} \big). \]
Thus we find
\begin{align*}
&\frac{l(\tau)^2-l(t)^2}{2(\tau-t)} \\
&\ge \int_{\R^n \times \R^n} \int_{\widetilde{\Gamma} \times \widetilde{\Gamma}}
 g_{\zeta(t)-\xi(t)}\bigg( \zeta(t)-\xi(t),
 \frac{\zeta(\tau)-\zeta(t)}{\tau -t}-\frac{\xi(\tau)-\xi(t)}{\tau -t} \bigg)
 \,d\Pi_x^{\tau}(\xi) d\Xi_y^{\tau}(\zeta) d\pi_{\tau}(x,y).
\end{align*}
Recall that $(\xi(\tau)-\xi(t))/(\tau -t)$ converges to
$\dot{\xi}(t)=[\nabla(-\rho_t)/\rho_t](\xi(t))$ uniformly on $\Omega_{\ve}$.
Moreover,
\[ d\tilde{\pi}_t^{\tau}:=(e_t \times e_t)_{\sharp} \bigg[
 \int_{\R^n \times \R^n} d\Pi_x^{\tau} d\Xi_y^{\tau} d\pi_{\tau}(x,y)
 \bigg] \in \cP_2(\R^n \times \R^n) \]
weakly converges to $\pi_t$ as $\tau \downarrow t$ due to \cite[Lemma~10.2.8]{AGS}.
Precisely, as $\tilde{\pi}_t^{\tau}$ is a coupling of $\mu_t$ and $\nu_t$,
the family $\{ \tilde{\pi}_t^{\tau} \}_{\tau \in (t,t+\delta)}$
is relatively compact (cf.\ \cite[Remark~5.2.3]{AGS}).
Combining this with the simple estimate
\begin{align*}
&\bigg( \int_{\R^n \times \R^n} \|y-x\|^2 \,d\tilde{\pi}_t^{\tau}(x,y) \bigg)^{1/2} \\
&\le \bigg( \int_{\widetilde{\Gamma}} \|\zeta(t)-\zeta(\tau)\|^2 \,d\Xi(\zeta) \bigg)^{1/2}
 +W_2(\mu_{\tau},\nu_{\tau})
 +\bigg( \int_{\widetilde{\Gamma}} \|\xi(\tau)-\xi(t)\|^2 \,d\Pi(\xi) \bigg)^{1/2} \\
&\to W_2(\mu_t,\nu_t) \qquad (\tau \downarrow t)
\end{align*}
and the uniqueness of the optimal coupling $\pi_t$,
we see that the limit of $\tilde{\pi}_t^{\tau}$ must be $\pi_t$.
Therefore we obtain
\[ \liminf_{\tau \downarrow t} \frac{l(\tau)^2-l(t)^2}{2(\tau-t)}
 \ge \int_{\R^n \times \R^n} g_{y-x}\bigg( y-x,
 \frac{\nabla(-\sigma_t)}{\sigma_t}(y)-\frac{\nabla(-\rho_t)}{\rho_t}(x) \bigg)
 \,d\pi_t(x,y) \]
and complete the proof.
$\qedd$
\end{proof}

Now, the following is shown in a similar way to Theorem~\ref{th:equiv}.

\begin{proposition}[Skew convexity versus Wasserstein contraction]\label{pr:Wcont}
For $K \in \R$, the following are equivalent.

\begin{enumerate}[{\rm (I)}]
\item The relative entropy is $K$-skew convex in the sense that,
for any $\mu=\rho \,dx \in \cP_2^{\ac}(\R^n)$ such that
$\rho \in H^1_0(\R^n) \cap C^1(\R^n)$, $\Ent(\mu)<\infty$,
$\|\nabla(-\rho)/\rho\| \in L^2(\R^n,\mu)$ and for any $\nu=\sigma \,dx \in \cP_2^{\ac}(\R^n)$
satisfying the same conditions, it holds
\begin{equation}\label{eq:omega}
\int_{\R^n} g_{\dot{\omega}_1}
 \big( \dot{\omega}_1,\nabla_W (-\Ent)(\nu) \big) \,d\nu
 -\int_{\R^n} g_{\dot{\omega}_0}
 \big( \dot{\omega}_0,\nabla_W (-\Ent)(\mu) \big) \,d\mu
 \le -KW_2(\mu,\nu)^2,
\end{equation}
where $\omega:[0,1] \lra \cP_2(\R^n)$ is the minimal geodesic from $\mu$ to $\nu$.

\item The reverse heat flow is $K$-contractive in the sense that, for any global solutions
$(\rho_t)_{t \ge 0}, (\sigma_t)_{t \ge 0} \subset H^1_0(\R^n)$
to the reverse heat equation such that
$\mu_t:=\rho_t \,dx \in \cP_2^{\ac}(\R^n)$ and
$\nu_t:=\sigma_t \,dx \in \cP_2^{\ac}(\R^n)$, we have
\begin{equation}\label{eq:Wcont}
W_2(\mu_t,\nu_t) \le e^{-Kt}W_2(\mu_0,\nu_0) \qquad
{\it for\ all}\ t \in [0,\infty).
\end{equation}
\end{enumerate}
\end{proposition}

We remark that, in the implication (II) $\Rightarrow$ (I),
we need the $C^1$-regularity of $\rho$ and $\sigma$ to apply
(the proof of) Proposition~\ref{pr:W1vf} at $t=0$.
Indeed, as each spatial derivative $v:=\del\rho/\del x^k$ ($k=1,2,\ldots,n$)
again solves the linear parabolic equation $\del_t v=\Delta^{\nabla(-\rho)}v$ (\cite[Lemma~4.7]{OS}),
the upper Gaussian estimate for the fundamental solution
ensures that $\del\rho_t/\del x^k$ tends to $\del\rho_0/\del x^k$ locally uniformly.

As a corollary to Proposition~\ref{pr:Wcont}, we obtain the $0$-contraction
of gradient curves in a special class of \emph{symmetric} measures
(compare this with Step~\ref{st:lp} in the next section).

\begin{corollary}[Non-expansion for Gaussian measures]\label{cr:Wcont}
Let $(\R^n,\|\cdot\|)$ be a symmetric normed space $($i.e., $\|\!-x\|=\|x\|)$.
Take two probability measures of Gaussian form
\[ d\mu_0(x)=Ca^{-n/2}\exp\bigg(\! -\frac{\|x-y\|^2}{4a} \bigg) dx, \quad
 d\nu_0(x)=Cb^{-n/2}\exp\bigg(\! -\frac{\|x-z\|^2}{4b} \bigg) dx \]
for some $a,b>0$, $y,z \in \R^n$ and the normalizing constant $C>0$.
Then the gradient curves $(\mu_t)_{t \ge 0}$, $(\nu_t)_{t \ge 0}$ of $\Ent$
starting from them satisfies $W_2(\mu_t,\nu_t) \le W_2(\mu_0,\nu_0)$
for all $t \ge 0$.
\end{corollary}

\begin{proof}
Without loss of generality, we assume $y=0$ and $a \ge b$.
Solving the heat equation, we observe
\[ d\mu_t(x)=C(t+a)^{-n/2}\exp\bigg(\! -\frac{\|x\|^2}{4(t+a)} \bigg) dx, \
 d\nu_t(x)=C(t+b)^{-n/2}\exp\bigg(\! -\frac{\|x-z\|^2}{4(t+b)} \bigg) dx. \]
Note also that the unique minimal geodesic $(\omega_s)_{s \in [0,1]}$
from $\mu_t$ to $\nu_t$ is given by $(T_s)_{\sharp}\mu_t$, where
\[ T_s(x):=(1-s)x +s\bigg( z+\sqrt{\frac{t+b}{t+a}}x  \bigg). \]
We can explicitly write as
\[ d\omega_s(x)=C\big(t+(1-s)a+sb\big)^{-n/2}
 \exp\bigg(\! -\frac{\|x-sz\|^2}{4(t+(1-s)a+sb)} \bigg) dx. \]
It follows from Lemma~\ref{lm:DEnt} that
\begin{align*}
&\int_{\R^n} g_{\dot{\omega}_s}\big( \dot{\omega}_s,
 \nabla_W(-\Ent)(\omega_s) \big) \,d\omega_s \\
&= \int_{\R^n} g_{T_1(x)-x}\big( T_1(x)-x,
 \big[ \nabla_W(-\Ent)(\omega_s) \big] (T_s(x)) \big) \,d\mu_t(x) \\
&= \frac{1}{2(t+(1-s)a+sb)} \int_{\R^n}
 g_{T_1(x)-x}\big( T_1(x)-x,T_s(x)-sz \big) \,d\mu_t(x) \\
&= -\frac{(1-s)+s\sqrt{(t+b)/(t+a)}}{2(t+(1-s)a+sb)}
 \int_{\R^n} g_{T_1(x)-x}\big( x-T_1(x),x \big) \,d\mu_t(x).
\end{align*}
Observe that the coefficient of the last line is non-increasing in $s$.
Hence it suffices to show that
$\Theta:=\int_{\R^n} g_{T_1(x)-x}(x-T_1(x),x) \,d\mu_t(x)$
(which is independent of $s$) is nonnegative.
If $a=b$, then we find $T_1(x)-x \equiv z$ and $\Theta=0$
by the symmetry of $\mu_t$.
If $a>b$, then we put $z'=s'z=\{ 1-\sqrt{(t+b)/(t+a)} \}^{-1} z$
and deduce $T_{s'}(x) \equiv z'$ (note that $s'>1$).
Thus we have
\[ g_{T_1(x)-x}\big( x-T_1(x),x \big) =\frac{1}{s'}g_{z'-x}(x-z',x)
 =\frac{1}{2s'}[D(\|z'-\cdot\|^2)(x)](x), \]
and $[D(\|z'-\cdot\|^2)(x)](x)+[D(\|z'-\cdot\|^2)(-x)](-x) \ge 0$
by the convexity of $\|z'-\cdot\|^2$ (along with the symmetry of $\|\cdot\|$).
Therefore we obtain $\Theta \ge 0$, and Proposition~\ref{pr:Wcont}
completes the proof.
$\qedd$
\end{proof}

\section{Non-contraction of heat flow}\label{sc:nonc}

This section is devoted to a proof of Theorem~\ref{th:nonc}.
For notational simplicity, we prove this for the reverse norm.
That is to say, global solutions to the reverse heat equation \eqref{eq:heat}
are not $K$-contractive with respect to $W_2$.

Fix $\mu=\rho \,dx \in \cP^{\ac}_2(\R^n)$ such that
$\rho \in H^1_0(\R^n) \cap C^1(\R^n)$, $\Ent(\mu)<\infty$
and that $\|\nabla(-\rho)\|/\rho \in L^2(\R^n,\mu)$.
For $T >1$, we set
\[ \omega_s=\rho^s \,dx:=(\cF_{s/T})_{\sharp}\mu, \quad
 \cF_{s/T}(x):=\bigg( 1-\frac{s}{T} \bigg) x\quad
 {\rm for}\ s \in [0,T]. \]
Then $(\omega_s)_{s \in [0,T]}$ is the unique minimal geodesic from $\mu$
to the Dirac measure $\delta_O$ at the origin $O \in \R^n$,
and its tangent vector field $\dot{\omega}_s$ is simply given by
$\dot{\omega}_s(x)=-x/(T-s)$.
Put $\nu=\omega_1$.
We will show that \eqref{eq:omega} is false for any given $K \in \R$
(i.e., $\Ent$ is not $K$-skew convex) by choosing suitable $\rho$.

We deduce from $\nabla_W(-\Ent)(\omega_s)=\nabla(-\rho^s)/\rho^s$
(Lemma~\ref{lm:DEnt}) that
\[ \int_{\R^n} g_{\dot{\omega}_s}
 \big( \dot{\omega}_s,\nabla_W(-\Ent)(\omega_s) \big) \,d\omega_s
 =\frac{1}{T-s} \int_{\supp\rho^s} g_{-x}\big(\! -x,\nabla(-\rho^s)(x) \big) \,dx. \]
It follows from $\rho^s(x)=(T/(T-s))^n \rho(Tx/(T-s))$
and the change of variables formula that
\begin{align*}
\int_{\supp\rho^s} g_{-x}\big( x,\nabla(-\rho^s)(x) \big) \,dx
&= \int_{\supp\rho^s} g_{-x}\bigg( x,\bigg( \frac{T}{T-s} \bigg)^{n+1}
 \nabla(-\rho)\bigg( \frac{Tx}{T-s} \bigg) \bigg) \,dx \\
&= \int_{\supp\rho} g_{-x}\big( x,\nabla(-\rho)(x) \big) \,dx.
\end{align*}
Thus we have
\[ \int_{\R^n} g_{\dot{\omega}_s}
 \big( \dot{\omega}_s,\nabla_W(-\Ent)(\omega_s) \big) \,d\omega_s
 =\frac{1}{T-s} \int_{\supp\rho} g_{-x}\big(\! -x,\nabla(-\rho)(x) \big) \,dx, \]
and hence
\[ \frac{d}{ds} \bigg[ \int_{\R^n} g_{\dot{\omega}_s}
 \big( \dot{\omega}_s,\nabla_W(-\Ent)(\omega_s) \big) \,d\omega_s \bigg]
 = \frac{1}{(T-s)^2} \int_{\supp\rho}g_{-x}\big(\! -x,\nabla(-\rho)(x) \big) \,dx. \]
Note that
$W_2(\omega_0,\omega_1)^2=T^{-2}\int_{\R^n} \|\!-x\|^2 \rho(x) \,dx$
and, by putting $f(x):=\|\!-x\|^2/2$ (as in Section~\ref{sc:dist}),
\[ g_{-x}\big(\! -x,\nabla(-\rho)(x) \big) =[D(-f)(x)]\big( \nabla(-\rho)(x) \big). \]
We set
\begin{equation}\label{eq:nonc}
\Theta(\rho) :=
 \frac{\int_{\supp\rho} [D(-f)(x)](\nabla(-\rho)(x))\,dx}
 {\int_{\R^n} \|\!-x\|^2 \rho(x) \,dx}
\end{equation}
and shall demonstrate that $\Theta(\rho)$ can be positive
(Steps~\ref{st:lp}--\ref{st:ndim} below) and even arbitrarily large
(Step~\ref{st:scal}) by choosing suitable $\rho$, unless $\|\cdot\|$
is an inner product.
This means that \eqref{eq:omega} is false for any $K \in \R$,
and completes the proof of Theorem~\ref{th:nonc}.
We start with an explicit example describing the heart of our construction.

\setcounter{step}{-1}
\begin{step}[The model case of $\ell^2_p$ with $2<p<\infty$]\label{st:lp}
Let $\|\cdot\|$ be the $\ell_p$-norm of $\R^2$ such that $2<p<\infty$.
Take the unit vectors $a=(-1,0)$, $b=(2^{-1/p},2^{-1/p})$,
$c=(2^{-1/p},-2^{-1/p})$ and let $\triangle ABC$ be the triangle
tangent to the unit sphere of $\|\cdot\|$ at $a,b,c$.
Precisely, $A=(2^{1-1/p},0)$, $B=(-1,-1-2^{1-1/p})$ and $C=(-1,1+2^{1-1/p})$
(Figure~3).

\begin{center}
\begin{picture}(400,220)
\put(180,5){Figure~3}

\qbezier(165,87)(190,81)(215,87)
\qbezier(165,153)(190,159)(215,153)
\qbezier(223,145)(229,120)(223,95)
\qbezier(157,145)(151,120)(157,95)
\qbezier(165,87)(159,89)(157,95)
\qbezier(165,153)(159,151)(157,145)
\qbezier(215,87)(221,89)(223,95)
\qbezier(215,153)(221,151)(223,145)

\put(189,119){\rule{2pt}{2pt}}
\put(153,119){\rule{2pt}{2pt}}
\put(219,149){\rule{2pt}{2pt}}
\put(219,89){\rule{2pt}{2pt}}

\thicklines
\put(154,24){\line(0,1){192}}
\qbezier(154,24)(220,90)(250,120)
\qbezier(154,216)(220,150)(250,120)
\put(190,120){\line(1,0){60}}
\qbezier(190,120)(172,72)(154,24)
\qbezier(190,120)(172,168)(154,216)

\put(190,120){\vector(1,1){20}}
\put(190,120){\vector(1,-1){20}}
\put(190,120){\vector(-1,0){24}}

\put(255,120){$A(2^{1-1/p},0)$}
\put(170,23){$B(-1,-1-2^{1-1/p})$}
\put(170,210){$C(-1,1+2^{1-1/p})$}
\put(105,120){$a(-1,0)$}
\put(230,150){$b(2^{-1/p},2^{-1/p})$}
\put(230,85){$c(2^{-1/p},-2^{-1/p})$}
\put(174,124){$O$}

\end{picture}
\end{center}

Define the nonnegative function $\hat{\rho}:\R^2 \lra [0,\infty)$ by
$\hat{\rho}:=0$ outside $\triangle ABC$ and by $\hat{\rho}(tx):=(1-t)\sigma$
for a point $x$ on the edges of $\triangle ABC$ and for $t \in [0,1]$,
where the constant $\sigma>0$ is chosen so that $\int_{\R^2} \hat{\rho} \,dx=1$.
Note that the gradient vector $\nabla(-\hat{\rho})$ is
$\sigma \cdot a=(-\sigma,0)$ inside $\triangle OBC$,
$\sigma \cdot b=(2^{-1/p}\sigma,2^{-1/p}\sigma)$ inside $\triangle OAC$,
and $\sigma \cdot c=(2^{-1/p}\sigma,-2^{-1/p}\sigma)$ inside $\triangle OAB$.
Hence we have
\begin{align}
\int_{\triangle ABC} \nabla(-\hat{\rho}) \,dx
&=(1+2^{1-1/p}) \cdot (-\sigma,0)
 +2^{1-1/p}(1+2^{1-1/p}) \cdot (2^{-1/p}\sigma,0) \nonumber\\
&=(1+2^{1-1/p})\sigma \cdot (2^{1-2/p}-1,0). \label{eq:lp}
\end{align}
Note that $2^{1-2/p}-1>0$ since $p>2$.

Now, for large $R>1$, we consider the function
$\hat{\rho}_R(x):=\hat{\rho}(x+(R,0))$.
Then it follows from \eqref{eq:lp} and $\nabla(-f)(x)=-x$ that
\[ \lim_{R \to \infty} \int_{\supp\hat{\rho}_R}
 \bigg[ D\bigg( \frac{-f}{R} \bigg) (x) \bigg] \big( \nabla(-\hat{\rho}_R)(x) \big) \,dx
 =(1+2^{1-1/p})(2^{1-2/p}-1)\sigma >0. \]
Therefore, by taking a smooth approximation of $\hat{\rho}_R$
(satisfying the conditions imposed on $\rho$ at the beginning of the section)
for sufficiently large $R$, we find $\rho$ satisfying $\Theta(\rho)>0$.
\end{step}

\begin{step}[General two-dimensional case]\label{st:2dim}
The argument in Step~\ref{st:lp} shows the following claim.
We will denote by $\bS(1)$ the unit sphere of the norm $\|\cdot\|$.

\begin{claim}\label{cl:nonc}
Suppose that a Minkowski space $(\R^2,\|\cdot\|)$ admits a triangle
$\triangle ABC$ such that edges $AB, BC, CA$ are tangent to
$\bS(1)$ at points $c, a, b$, respectively, and that the vector
\begin{equation}\label{eq:ABC}
|\triangle OAB| \cdot c +|\triangle OBC| \cdot a +|\triangle OCA| \cdot b
\end{equation}
is nonzero, where $|\triangle OAB|$ denotes the area of $\triangle OAB$
with respect to the Lebesgue measure.
Then there exists a function $\rho$ for which $\Theta(\rho)$
as in \eqref{eq:nonc} is positive.
\end{claim}

Note that \eqref{eq:ABC} is always zero in inner product spaces.
(Indeed, for the standard inner product, it holds that
$\langle |\triangle OAB| \cdot c +|\triangle OBC| \cdot a
 +|\triangle OCA| \cdot b, e_i \rangle=0$
for $e_1=(1,0)$ and $e_2=(0,1)$ by the fundamental theorem of calculus
applied to the function $\hat{\rho}$ defined as in Step~\ref{st:lp}.)
Claim~\ref{cl:nonc} is sharp enough for our purpose,
as we can certainly verify the following.

\begin{claim}\label{cl:tri}
There exists a triangle $\triangle ABC$ satisfying the condition in
Claim~$\ref{cl:nonc}$ unless $\|\cdot\|$ is an inner product.
\end{claim}

Although this claim should be a known fact (and there would be a simpler proof),
we give a proof for completeness.
We first treat the easier case of nonsymmetric norms.
Choose a pair $a,b \in \bS(1)$ such that $b=-\lambda a$ with $\lambda \neq 1$.
If the tangent lines of $\bS(1)$ at $a$ and $b$ are not parallel in $\R^2$,
then we draw a triangle $\triangle ABC$ in such a way that
the edge $AB$ is parallel to $ba$.
As the vectors $a$ and $c$ are linearly independent,
\eqref{eq:ABC} is not zero.
In the other case where the tangent lines of $\bS(1)$ at $a$
and $b$ are parallel, we take $C'$ such that $OC'$ is parallel to
these tangent lines, and draw the triangle $\triangle A'B'C'$ such that
$A'B'$ is parallel to $b'a'$ ($a'$ and $b'$ are determined only by $C'$).
By letting $C'$ go to infinity, $a'$ and $b'$ can be arbitrarily close to
$a$ and $b$, respectively.
Then we observe that $|\triangle OA'B'|$ is much smaller than $|\triangle OB'C'|$
and $|\triangle OC'A'|$, and that the ratio $|\triangle OC'A'|/|\triangle OB'C'|$
is close to $\lambda$.
Thus we have
\[ |\triangle OA'B'| \cdot c' +|\triangle OB'C'| \cdot a' +|\triangle OC'A'| \cdot b'
 \approx |\triangle OB'C'|(1-\lambda^2) \cdot a \neq 0. \]

Next we consider symmetric norms.
We suppose that the sum \eqref{eq:ABC} is always zero,
and will see that $\|\cdot\|$ must be an inner product.
Take $a,b \in \bS(1)$ with $b=-a$ such that $|a|=\sup_{x \in \bS(1)}|x|$,
where $|\cdot|$ is the Euclidean norm.
Then the tangent lines of $\bS(1)$ at $a$ and $b$ are perpendicular
to $ab$ with respect to the Eulidean inner product.
As in the nonsymmentric case, we take $C$ so that $OC$
is parallel to these tangent lines, and consider the triangle $\triangle A'B'C$
for some fixed $c$.
Let $C$ diverge to infinity and denote the limits of $A',B'$ by $A,B$.
Then our hypothesis yields that the vector
\begin{equation}\label{eq:abc}
|\triangle OaB| \cdot a +|\triangle OAb| \cdot b +|\triangle OAB| \cdot c
\end{equation}
is independent of the choice of $c$ on the arc between $a$ and $b$ opposite to $C$
(since $\triangle a'b'C$ corresponding to $\triangle A'B'C$ was independent of the choice of $c$).
We will see that this is the case only for inner products.
For simplicity, we assume that $a=(-1,0)$, $b=(1,0)$ and that
$c$ is in the upper half plane.
Define the function $h:[-1,1] \lra [0,1]$ by $\|(t,h(t))\| \equiv 1$,
and compare this with the function $\tilde{h}:[-1,1] \lra [0,1]$
such that $\{(t,\tilde{h}(t))\}_{t \in [-1,1]}$ draws (the upper half of)
the ellipse having $ab$ and $OD_0$ as its long and short axes,
where $D_0=(0,\sup h)$ (Figure~4).
We first suppose that $\sup h$ is attained at $t_0>0$, and put
$c_0=(t_0,h(t_0))$, $A_0=(1,h(t_0))$ and $B_0=(-1,h(t_0))$.
Then, on the one hand, clearly the $y$-components of the vectors 
$|\triangle OA_0B_0| \cdot c_0$ and $|\triangle OA_0B_0| \cdot D_0$ are the same.
On the other hand, since only $c$ has a nonzero $y$-component in \eqref{eq:abc},
$|\triangle OAB| \cdot c$ and $|\triangle OA_0 B_0| \cdot c_0$ have the same $y$-component.
Similarly, for any $t' \in (-1,1)$ and points $D:=(t',\tilde{h}(t'))$,
$\widetilde{A}:=(1,\tilde{h}(t')+(1-t')\tilde{h}'(t'))$ and
$\widetilde{B}:=(-1,\tilde{h}(t')-(1+t')\tilde{h}'(t'))$ corresponding to the ellipse
drawn by $\tilde{h}$, $|\triangle O\widetilde{A}\widetilde{B}| \cdot D$ has
the same $y$-component as $|\triangle OA_0 B_0| \cdot D_0$.
Hence we have
\[ |\triangle OAB| \cdot c
 -|\triangle O\widetilde{A}\widetilde{B}| \cdot D \in \R \times \{0\}. \]

\begin{center}
\begin{picture}(400,200)
\put(180,5){Figure~4}

\put(50,30){\line(1,0){300}}
\put(100,100){\line(1,0){200}}
\put(100,30){\line(0,1){150}}
\put(200,30){\line(0,1){150}}
\put(300,30){\line(0,1){150}}
\put(100,90){\line(1,0){200}}

\qbezier(100,30)(110,95)(200,100)
\qbezier(200,100)(290,95)(300,30)
\qbezier(100,149)(200,110)(300,71)

\put(199,99){\rule{2pt}{2pt}}
\put(251,89){\rule{2pt}{2pt}}

\thicklines
\qbezier(100,30)(105,50)(120,65)
\qbezier(120,65)(130,75)(160,88)
\qbezier(160,88)(195,100)(220,100)
\qbezier(220,100)(230,100)(267,90)
\qbezier(267,90)(300,75)(300,30)

\qbezier(100,140)(200,110)(300,80)

\put(219,99){\rule{2pt}{2pt}}
\put(266,89){\rule{2pt}{2pt}}

\put(188,35){$O$}
\put(127,55){$h$}
\put(110,75){$\tilde{h}$}
\put(305,100){$A_0$}
\put(305,78){$A$}
\put(305,60){$\widetilde{A}$}
\put(83,100){$B_0$}
\put(85,137){$B$}
\put(85,153){$\widetilde{B}$}
\put(268,93){$c$}
\put(245,77){$D$}

\put(97,17){$a$}
\put(297,17){$b$}
\put(220,110){$c_0$}
\put(183,103){$D_0$}

\end{picture}
\end{center}

In particular, for any $c=(t,h(t))$ with $t \in (t_0,1)$ and $t' \in (0,1)$ with
$\tilde{h}(t')=h(t)$, we obtain
$|\triangle OAB|=|\triangle O\widetilde{A} \widetilde{B}|$ and hence
$\tilde{h}'(t')<h'(t)$ (more precisely, $AB$ and $\widetilde{A}\widetilde{B}$
must intersect on the $y$-axis).
However, this is a contradiction since $\tilde{h}(1)=h(1)=0$.
We similarly derive a contradiction from $t_0<0$, so that $t_0=0$.
Furthermore, $h$ must coincide with $\tilde{h}$ everywhere by a similar discussion.
Therefore $\|\cdot\|$ is an inner product and we complete the proof of
Claim~\ref{cl:tri}.
\end{step}

\begin{step}[$n$-dimensional case with $n \ge 3$]\label{st:ndim}
Suppose that $(\R^n,\|\cdot\|)$ is not an inner product space.
Then there is a two-dimensional subspace $P \subset \R^n$ in which
the restriction of $\|\cdot\|$ is not an inner product.
We assume $P=\{ (x,y,0,\ldots,0) \,|\, x,y \in \R \}$ for brevity,
and sometimes identify this with $\R^2$.
By Step~\ref{st:2dim}, there is a function
$\rho_R:(\R^2,\|\cdot\| |_P) \lra [0,\infty)$ such that $\int_{\R^2} \rho_R \,dx=1$,
$\supp\rho_R \subset B((-R,0),r)$ for some fixed $r>0$ and that
\[ \lim_{R \to \infty} \int_{\supp \rho_R} \bigg[ D\bigg( \frac{-f}{R} \bigg) (x) \bigg]
 \big( \nabla(-\rho_R)(x) \big) \,dx >0, \]
where we set $B(z,r):=\{w \in \R^2 \,|\, \|w-z\|<r \}$ for $z \in \R^2$ and $r>0$.
Using a smooth cut-off function $\eta_R:\R^{n-2} \lra [0,\infty)$
such that $\eta_R \equiv 1$ on $B(O,\sqrt{R})$, $\supp\eta_R \subset B(O,\sqrt{R}+1)$
and that $\sup\|\nabla(-\eta_R)\| <2$, define $\rho:(\R^n,\|\cdot\|) \lra [0,\infty)$ by
\[ \rho(x,y):=\bigg( \int_{\R^{n-2}}\eta_R \,dz \bigg)^{-1}\rho_R(x) \eta_R(y) \]
for $x \in \R^2$ and $y \in \R^{n-2}$.
Note that
$\nabla(-\rho)(x,y)=(\int_{\R^{n-2}}\eta_R\,dz)^{-1} \cdot \nabla(-\rho_R)(x)$
for $y \in B(O,\sqrt{R}) \subset \R^{n-2}$.
Hence we have $\Theta(\rho)>0$ since the effect of the boundary
of the cut-off is negligible for large $R$.
Indeed, we observe that
\[ |B(O,\sqrt{R}+1) \setminus B(O,\sqrt{R})| \cdot
 \bigg( \int_{\R^{n-2}}\eta_R\,dz \bigg)^{-1}
 =O\big( (\sqrt{R})^{n-3} / (\sqrt{R})^{n-2} \big) \to 0 \]
as $R$ goes to infinity.
\end{step}

\begin{step}[Scaling]\label{st:scal}
Suppose that there is $\rho$ with $\Theta(\rho)>0$, and set
$\rho_{\ve}(x):=\ve^{-n} \rho(\ve^{-1}x)$ for $\ve>0$.
Then we have
\begin{align*}
\int_{\supp\rho_{\ve}}
 [D(-f)(x)]\big( \nabla(-\rho_{\ve})(x) \big) \,dx
&= \ve^{-(n+1)} \int_{\supp\rho_{\ve}}
 [D(-f)(x)]\big( \nabla(-\rho)(\ve^{-1}x) \big) \,dx \\
&= \int_{\supp\rho} [D(-f)(x)]\big( \nabla(-\rho)(x) \big) \,dx
\end{align*}
and
\[ \int_{\R^n} \|\!-x\|^2 \rho_{\ve}(x) \,dx
 =\ve^{-n} \int_{\R^n} \|\!-x\|^2 \rho(\ve^{-1}x) \,dx
 =\ve^2 \int_{\R^n} \|\!-x\|^2 \rho(x) \,dx. \]
Therefore $\Theta(\rho_{\ve})=\ve^{-2}\Theta(\rho)$ and
it diverges to infinity as $\ve$ tends to zero.
Thus we complete the proof of Theorem~\ref{th:nonc}.
\end{step}

\renewcommand{\thesection}{\Alph{section}}
\setcounter{section}{1}
\setcounter{theorem}{0}
\setcounter{equation}{0}
\section*{Appendix: Skew convexity of distance functions on Finsler manifolds}

We finally investigate the skew convexity of squared distance functions on Finsler manifolds.
This would be of independent interest from the geometric viewpoint.
The convexity of distance functions is closely related to upper bounds
of the sectional curvature in the Riemannian case.
In our Finsler setting, we need two more quantities to control the distance function.
See \cite{Shlec} and \cite{Ouni} for related work on the usual convexity
and concavity along geodesics.

Let $(M,F)$ be a $C^{\infty}$-Finsler manifold.
We introduce some terminologies for which we refer to \cite{BCS}.
For a $C^1$-vector field $X$ on $M$ and tangent vectors $v,w \in T_xM$ with $w \neq 0$,
we define the \emph{covariant derivative} of $X$ by $v$ with \emph{reference vector} $w$ as
\[ (D^w_v X)(x):=\sum_{i,j=1}^n \bigg\{ v^j \frac{\del X^i}{\del x^j}(x)
 +\sum_{k=1}^n \Gamma^i_{jk}(w) v^j X^k(x) \bigg\}
 \frac{\del}{\del x^i}\Big|_x, \]
where $\Gamma^i_{jk}$ is the Christoffel symbol.
If $\Gamma^i_{jk}(w)$ depends only on the point $x$ (i.e., independent
of the choice of $w \in T_xM \setminus \{0\}$) for all $x \in M$,
then we say that $(M,F)$ is of \emph{Berwald type}.
In a Berwald space, all tangent spaces are isometric to each other.
For instance, Riemannian manifolds and Minkowski spaces are of Berwald type.

By using the covariant derivative, the geodesic equation is written
in a canonical way as
$D^{\dot{\gamma}}_{\dot{\gamma}} \dot{\gamma} \equiv 0$.
We will use the following formula borrowed from
\cite[Exercise~10.1.2]{BCS}:
\begin{equation}\label{eq:cov}
\frac{d}{dt}g_V(V,W)
 =g_V(D_{\dot{\gamma}}^V V,W) +g_V(V,D_{\dot{\gamma}}^V W)
\end{equation}
for any $C^1$-curve $\gamma$ and $C^1$-vector fields $V,W$ along $\gamma$
such that $V \not\equiv 0$.

A $C^{\infty}$-vector field $V$ along a geodesic $\gamma:[0,l] \lra M$
is called a \emph{Jacobi field} if it satisfies the equation
$D^{\dot{\gamma}}_{\dot{\gamma}} D^{\dot{\gamma}}_{\gamma}V
 +\cR(V,\dot{\gamma})\dot{\gamma} \equiv 0$,
where $\cR:TM \otimes TM \lra T^*M \otimes TM$ is the curvature tensor.
Similarly to the Riemannian case, the variational vector field of a geodesic variation
is a Jacobi field (and vice versa).
For linearly independent vectors $v,w \in T_xM$, the \emph{flag curvature}
is defined by
\[ \cK(v,w):=\frac{g_v(\cR(w,v)v,w)}{F(v)^2 g_v(w,w)-g_v(v,w)^2}. \]
We remark that $\cK(v,w)$ depends not only on the plane in $T_xM$ spanned by
$v$ and $w$ (\emph{flag}), but also on the choice of $v$ in it (\emph{flagpole}).

In order to state our theorem, we introduce the condition
\begin{equation}\label{eq:delta}
g_V(V,D^V_W D^V_W V -D^W_W D^W_W V) \ge -\delta F(V)^2 F(W)^2
\end{equation}
for non-vanishing $C^{\infty}$-vector fields $V,W$ and $\delta \ge 0$.
This clearly holds with $\delta=0$ for Berwald spaces.
Therefore $\delta$ measures how the tangent spaces are distorted
as one moves (in $M$) along $W$.
The \emph{injectivity radius} $\inj(z)$ at $z \in M$ is the supremum of $R>0$
such that any unit speed geodesic $\gamma:[0,R) \lra M$ with
$\gamma(0)=z$ contains no cut point of $z$.
We set $B(x,r):=\{ y \in M \,|\, d(x,y)<r \}$ for $x \in M$ and $r>0$.

\begin{theorem}\label{th:dist}
Let $(M,F)$ be a forward complete Finsler manifold and
suppose that $\cK \le k$, $\cS \le S$ and \eqref{eq:delta} hold
for some $k \ge 0$, $S \ge 1$ and $\delta \ge 0$.
Then the function $f(x):=d(x,z)^2/2$ is $K(k,S,\delta,r)$-skew convex
in $\overleftarrow{B}(z,r)$ for all $z \in M$ and $r \in (0,R)$, where we set
\[ K(k,S,\delta,r)
 :=\sqrt{kS^2+\delta}r \cdot \cot(\sqrt{kS^2+\delta}r) \]
and $R:=\min\{ \overleftarrow{\inj}(z),\pi/\sqrt{kS^2+\delta} \}$.
In particular, if $\cK \le 0$, then $f$ is
$(\sqrt{\delta}r \cot(\sqrt{\delta}r))$-skew convex in $\overleftarrow{B}(z,r)$
for $r \in (0,\min\{ \overleftarrow{\inj}(z),\pi/\sqrt{\delta} \})$ regardless $\cS$.
\end{theorem}

\begin{proof}
Fix a unit speed minimal geodesic $\gamma:[0,l] \lra \overleftarrow{B}(z,r)$
with $r<R$, and let $\sigma:[0,l] \times [0,1] \lra M$ be the
$C^{\infty}$-variation such that $\sigma_s:=\sigma(s,\cdot)$
is the unique minimal geodesic from $\gamma(s)$ to $z$.
Put $\cT(s,t):=\del_t\sigma(s,t)$ and $\cV(s,t):=\del_s\sigma(s,t)$.
Observe that $\dot{\gamma}(s)=\cV(s,0)$ and $\nabla(-f)(\gamma(s))=\cT(s,0)$.
Hence we need to bound the following:
\begin{equation}\label{eq:d1}
\frac{\del}{\del s}\big[ g_{\cV}\big( \cV(s,0),\cT(s,0) \big) \big]
 =g_{\cV}\big( \cV(s,0),D_s^{\cV} \cT(s,0) \big).
\end{equation}
We used \eqref{eq:cov} and the geodesic equation $D_s^{\cV}\cV(s,0) \equiv 0$.
As $D_s^{\cV} \cT=D_t^{\cV} \cV$ (cf.\ \cite[Exercise~5.2.1]{BCS}),
we deduce from \eqref{eq:cov} that
\begin{equation}\label{eq:d2}
g_{\cV}\big( \cV(s,0),D_s^{\cV} \cT(s,0) \big)
 =\frac{1}{2} \frac{\del}{\del t}[g_{\cV}(\cV,\cV)](s,0)
 =\frac{1}{2} \frac{\del [F(\cV)^2]}{\del t}(s,0).
\end{equation}

Again due to \eqref{eq:cov}, we observe
\begin{align*}
\frac{\del^2[F(\cV)]}{\del t^2}
&= \frac{\del}{\del t} \bigg[ \frac{g_{\cV}(\cV,D^{\cV}_t \cV)}{F(\cV)} \bigg] \\
&= \frac{g_{\cV}(\cV,D^{\cV}_t D^{\cV}_t \cV)
 +g_{\cV}(D^{\cV}_t \cV,D^{\cV}_t \cV)}{F(\cV)}
 -\frac{g_{\cV}(\cV,D^{\cV}_t \cV)^2}{F(\cV)^3} \\
&= \frac{g_{\cV}(\cV,D^{\cV}_t D^{\cV}_t \cV)}{F(\cV)}
 +\frac{F(\cV)^2 g_{\cV}(D^{\cV}_t \cV,D^{\cV}_t \cV)
  -g_{\cV}(\cV,D^{\cV}_t \cV)^2}{F(\cV)^3}.
\end{align*}
The second term is nonnegative by the Cauchy-Schwarz inequality.
Moreover, by the assumption \eqref{eq:delta}, we have
\[ g_{\cV}(\cV,D^{\cV}_t D^{\cV}_t \cV)
 \ge g_{\cV}(\cV,D^{\cT}_t D^{\cT}_t \cV) -\delta F(\cV)^2 F(\cT)^2. \]
Since $\cV(s,\cdot)$ is a Jacobi field, it holds
$D_t^{\cT}D_t^{\cT} \cV =-\cR(\cV,\cT)\cT$ and hence
\begin{align*}
\frac{\del^2[F(\cV)]}{\del t^2}
&\ge -\cK(\cT,\cV) \frac{F(\cT)^2 g_{\cT}(\cV,\cV) -g_{\cT}(\cT,\cV)^2}{F(\cV)}
 -\delta F(\cV) F(\cT)^2 \\
&\ge -k \frac{F(\cT)^2 g_{\cT}(\cV,\cV) -g_{\cT}(\cT,\cV)^2}{F(\cV)}
 -\delta F(\cV) F(\cT)^2.
\end{align*}
As $k \ge 0$, it follows from $\cS \le S$ that (recall \eqref{eq:CS})
\[ -k \{F(\cT)^2 g_{\cT}(\cV,\cV) -g_{\cT}(\cT,\cV)^2 \}
 \ge -k F(\cT)^2 g_{\cT}(\cV,\cV) \ge -k S^2 F(\cV)^2 F(\cT)^2. \]
Hence we obtain, together with $F(\cT) \le r$,
\[ \frac{\del^2[F(\cV)]}{\del t^2} \ge -(kS^2+\delta) r^2 F(\cV). \]
The above inequality shows that the function
\[ \frac{\del[F(\cV)]}{\del t} \sin\big( \sqrt{kS^2+\delta}r(1-t) \big)
 -F(\cV)\frac{\del}{\del t}\big[ \sin\big( \sqrt{kS^2+\delta}r (1-t) \big) \big] \]
is non-decreasing in $t \in [0,1]$, so that it is nonpositive for all $t$.
Thus we have
\begin{equation}\label{eq:d3}
\frac{1}{2} \frac{\del[F(\cV)^2]}{\del t}
 =F(\cV) \frac{\del [F(\cV)]}{\del t}
 \le -\sqrt{kS^2+\delta}r \cdot \cot\big( \sqrt{kS^2+\delta}r(1-t) \big) F(\cV)^2.
\end{equation}

Combining \eqref{eq:d1}, \eqref{eq:d2}, \eqref{eq:d3} and
$F(\cV(s,0))=F(\dot{\gamma}(s))=1$, we conclude
\[ \frac{\del}{\del s}\big[ g_{\cV}\big( \cV(s,0),\cT(s,0) \big) \big]
 \le -\sqrt{kS^2+\delta}r \cdot \cot(\sqrt{kS^2+\delta}r). \]
This completes the proof.
$\qedd$
\end{proof}

Interestingly enough, what appeared in Theorem~\ref{th:dist} is not
the $2$-uniform convexity constant $\cC$, but the smoothness constant $\cS$.
Compare this with the usual convexity in \cite[Theorem~5.1]{Ouni}.
We finally state the Berwald case separately.

\begin{corollary}\label{cr:dist}
Let $(M,F)$ be forward complete and of Berwald type
and suppose that $\cK \le k$ and $\cS \le S$ hold for some $k \ge 0$ and $S \ge 1$.
Then the function $f(x):=d(x,z)^2/2$ is $(\sqrt{k}Sr \cot(\sqrt{k}Sr))$-skew convex
in $\overleftarrow{B}(z,r)$ for all $z \in M$,
$r \in (0,\min\{ \overleftarrow{\inj}(z),\pi/\sqrt{k}S\})$.
In particular, if $\cK \le 0$, then $f$ is $1$-skew convex in
$\overleftarrow{B}(z,\overleftarrow{\inj}(z))$ regardless $\cS$.
\end{corollary}

This recovers the $1$-skew convexity of $f(x)=\|\!-x\|^2/2$ on Minkowski spaces
in Section~\ref{sc:dist}.


\end{document}